\documentclass[reqno]{amsart}
\usepackage{amsfonts}
\usepackage{latexsym}
\usepackage{graphicx}
\usepackage{amsmath}
\usepackage{hyperref}
\usepackage{amssymb}
\numberwithin{equation}{section}

\usepackage{xcolor,color}

\usepackage{enumerate, hyperref}

\numberwithin{equation}{section}

\newtheorem{thm}{Theorem}[section]
\newtheorem{lem}[thm]{Lemma}

\theoremstyle{definition}
\newtheorem{definition}{Definition}[section]
\newtheorem{rem}[thm]{Remark}

\newcommand\R{{\mathbb R}}
\newcommand\C{{\mathbb C}}

\begin{document}

\title{On the unconditional uniqueness for NLS in $H^s$}

\author[Zheng Han]{Zheng Han $^1$}

\author[Daoyuan Fang]{Daoyuan Fang $^1$}

\address{$^{1}$
Department of Mathematics, Zhejiang University, Hangzhou,
310027, China}

\email[Zheng Han]{\href{mailto:hanzheng5400@yahoo.com.cn}{hanzheng5400@yahoo.com.cn}}
\email[Daoyuan Fang]{\href{mailto:dyf@zju.edu.cn}{dyf@zju.edu.cn}}



\begin{abstract}
In this article,  we study the unconditional uniqueness of $\dot
H^s$, $0<s< 1$, solutions for the nonlinear Schr\"o\-din\-ger
equation $i\partial _t u +\Delta u+ c |u|^\alpha u=0$  in ${\mathbb
R}^n$. We give a unified proof of the previously known results in
the subcritical cases and critical cases, and we also extend these
results to some previously unsettled cases. Our proof uses in
particular negative order Sobolev spaces (or Besov spaces), general
Strichartz estimates, and the improved regularity property for the
difference of two solutions.
\end{abstract}

\maketitle
\section{Introduction}
We study the uniqueness of $\dot{H}^s$ solutions of the following
Schr\"odinger equation:
\begin{equation}\label{NLS} \tag{NLS}
\left\{\begin{array}{ll}
       i\partial_t u+\Delta u+c|u|^\alpha u=0, \\
       u(0)=\varphi\in\dot{H}^s(\R^n), \  t\in[0,T], \ x\in\R^n, \
       n\geq2,
       \end{array}
       \right.
\end{equation}
where $\dot{H}^s$ is the homogeneous Sobolev space, $c\in\C$, $T>0$,
$\alpha>0$ and $s\in (0,\frac n2)$.

 To ensure that the initial value problem is locally well-posed in $H^s(\R^n)$, from Sobolev embedding,  one has to  assume $\alpha\le\frac 4{n-2s}$.   Furthermore, the equation \eqref{NLS} may not make sense, even in the sense of distribution, without an auxiliary space if $\alpha > \frac{n+2s}{n-2s}$. Therefore, one usually constructs the solution within the framework of $C([0,T]; H^s)\cap X$, where $X$ is an auxiliary space. For instance,
 Ginibre and Velo
(\cite{GINVEL}), Kato (\cite{Katou}), Cazenave and Weissler
(\cite{CWHs}) proved that \eqref{NLS} is locally well-posed in
\begin{equation}\label{auxiliary space}
C([0,T_{\text{max}});H^s)\cap
L^q_{\text{loc}}(0,T_{\text{max}};B_{r,2}^s),
\end{equation}
 where
$q=\frac{4(\alpha+2)}{\alpha(n-2s)},\
r=\frac{n(\alpha+2))}{n+s\alpha}$ and $B_{r,2}^s$ is the usual Besov
space.

 The uniqueness of solutions that belongs to an auxiliary space such as $L^q_tB^s_{r,2}$  as
well as $C([0,T];H^s)$ is called  conditional uniqueness. On the
other hand, the uniqueness without any auxiliary space is called
unconditional uniqueness. This problem, in the subcritical case, was
first studied by Kato  \cite{Katou}, in which the following results
are obtained:

The uniqueness holds in $C([0,T];H^s)$ if any of the following three
conditions is satisfied:
\begin{enumerate}
  \item $n=1,\ 0\leq s<\frac{1}{2},\ 0<\alpha<\frac{1+2s}{1-2s};$
  \item $n\geq2,\ 0\leq s<\frac{n}{2},\
  0<\alpha<\min\{\frac{4}{n-2s},
  \frac{2+2s}{n-2s}\};$
  \item $n\geq1,\ s\geq\frac{n}{2}.$
\end{enumerate}
From Kato's work, we can see that when $1\leq s$ and
$0<\alpha<\min\{\frac{4}{n-2s},\frac{n+2s}{n-2s}\}$, the
unconditional uniqueness holds.

Furioli and Terraneo \cite{FurioliT} extended Kato's results by
using negative order Besov spaces. They proved  uniqueness in the
slightly larger space $C([0,T];\dot{H}^s)$ when
\begin{equation}
n\geq3,\ \max\{1,\frac{2s}{n-2s}\}<\alpha<\min\{\frac{2+4s}{n-2s},\
\frac{4}{n-2s},\ \frac{n+2s}{n-2s},\ \frac{n+2-2s}{n-2s}\}.
\end{equation}

In \cite{Rogers}, Rogers applied a generalized  Strichartz estimate
(see \cite{Vilela}) to show that if
\begin{equation}
n\geq3,\
\frac{2+2s}{n-2s}\leq\alpha<\min\{\frac{2+4s(1-\frac{1}{n})}{n-2s},\
\frac{4}{n-2s}\},
\end{equation}
then  uniqueness is established in $C([0,T];\dot{H}^s)$.

Recently, Win and Tsutsumi  \cite{WinT} improved unconditional uniqueness in the dimension $3$ under the following
assumptions:
\begin{equation}
n=3,\ 1>s>\frac{1}{2},\ \max\{\frac{2+4s(1-\frac{1}{n})}{n-2s},\
\frac{n+2-2s}{n-2s}\}\leq\alpha<\min\{\frac{4}{n-2s},\frac{n+2s}{n-2s}\},
\end{equation}
where the initial datum belongs to $\dot{H}^s$.

In summary, in the subcritical case, the problem of
unconditional uniqueness is left open only when $0\leq s<1$, in the following three cases:

\begin{description}
  \item[Case a] $n=3,4,\ \frac{2+4s(1-\frac{1}{n})}{n-2s}\leq\alpha<1\text{
or
}\max\{\frac{2+4s}{n-2s},1\}\leq\alpha<\min\{\frac{n+2s}{n-2s},\frac{4}{n-2s}\};$
  \item[Case b] $n=5$,  $\frac{2+4s(1-\frac{1}{n})}{n-2s}\leq\alpha\leq1$ with $\alpha<\frac{4}{n-2s}$;
  \item[Case c] $n\geq6$,
  $\frac{2+4s(1-\frac{1}{n})}{n-2s}\leq\alpha<\frac{4}{n-2s}$.
\end{description}

For the critical case
$\alpha=\min\{\frac{n+2s}{n-2s},\frac{4}{n-2s}\}$, we can recall the
known results as follows: Kato firstly proved  unconditional
uniqueness in the dimension $1$ or when $s\geq\frac{n}{2}$ in
\cite{Katou}. Cazenave  \cite{CLN}(Proposition~4.2.13) showed that
when $1\leq s<n/2$ with $n\geq3$,  unconditional uniqueness still
holds. Win and Tsutsumi \cite{WinT} proved  unconditional uniqueness
in the following cases:
\begin{equation}
n=3,\ 1>s>\frac{1}{2},\ \alpha=\frac{4}{n-2s}\text{ and } n=4,5,\
1>s\geq\frac{1}{2},\ \alpha=\frac{4}{n-2s}.
\end{equation}
There are also some gaps for the critical case, especially when
$0<s<\frac{1}{2}$ or high dimensions. In particular, for $0\leq s<1$, the following cases are open:
\begin{description}
  \item[Case a] $n=2$, $\alpha=\frac{n+2s}{n-2s};$
  \item[Case b] $n=3$, $0\leq s\leq\frac{1}{2}$ with
  $\alpha=\frac{n+2s}{n-2s}$;
  \item[Case c] $n=4,5$, $0\leq s<\frac{1}{2}$ with $\alpha=\frac{4}{n-2s}$;
  \item[Case d] $n\geq6$,
  $\alpha=\frac{4}{n-2s}$.
\end{description}

From the above description, the authors in
\cite{Katou},\cite{FurioliT},\cite{Rogers} and \cite{WinT} apply
different methods to obtain various conclusions. The conclusions
they obtained overlap, but do not cover each other. In this article,
in addition to extending the known results to a larger domain of
indices, in particular the case $\alpha<1$, we also give a unified
proof of the results of \cite{FurioliT},\cite{Rogers},\cite{WinT}
and
 \cite{Katou} in either subcritical case or critical case.
Note that the nonlinearity is locally Lipschitz continuous when
$\alpha\geq1$, and locally H\"older continuous when $\alpha<1$. For
this reason, we have to use the different argument in cases $0\le
\alpha <1$ and $\alpha \ge 1$. Firstly, we show the results on the
subcritical case:
\begin{thm}[$\alpha\geq1$]\label{Bg1}
Let $0< s<1$, $n=3,4,5$, and assume
\begin{equation}
         \max\{1,\frac{2s}{n-2s}\}\leq\alpha<\min{\{\frac{4}{n-2s},\frac{n+2s}{n-2s},\frac{4s+4-n/(n-1)}{n-2s}\}}.
\end{equation}
Given $\varphi\in\dot{H}^s$ and  $T>0$, unconditional uniqueness
holds in $L^\infty(0,T;\dot{H}^s)$ for \eqref{NLS}.
\end{thm}

\begin{thm}[$\alpha<1$]\label{Ls1}
Let $0< s<1$ , $\varphi\in\dot{H}^s$ and
$T>0$. If $\alpha$ satisfies one of the following conditions, then
 unconditional uniqueness holds in $L^\infty(0,T;\dot{H}^s)$ for
\eqref{NLS}:
\begin{itemize}
  \item when $n=3$,
  \begin{equation}\left\{\begin{array}{ll}
                               \frac{2s}{3-2s}<\alpha<1,\ \ &\text{if
                               }s\geq\frac{3}{4},\\
                               \frac{2s}{3-2s}<\alpha<\min{\{1,\frac{2s+\frac{5}{2}}{3-4s}\}},&\text{if
                               }s<\frac{3}{4},\\
                               \end{array}
                               \right.
        \end{equation}
  \item  when $n\geq4$,
  \begin{equation}
  \frac{2s}{n-2s}<\alpha<\min{\{1,\frac{4}{n-2s},\frac{2s+4-\frac{n}{n-1}}{n-4s}\}}.
   \end{equation}
\end{itemize}
\end{thm}

\begin{rem}
According to our results, the following cases for unconditional
uniqueness are still left open for $0\leq s<1$:
\begin{description}
  \item[Case a] $n=3,4$, $\min\{\frac{4s+4-n/(n-1)}{n-2s},\frac{2s+4-n/(n-1)}{n-4s}\}\leq\alpha<\min\{\frac{4}{n-2s},\frac{n+2s}{n-2s}\}$;
  \item[Case b] $n\geq5$,
  $\frac{2s+4-n/(n-1)}{n-4s}\leq\alpha<\frac{4}{n-2s}$;
  \item[Case c] $n\geq3$, $s=0$,
  $\frac{2}{n}\leq\alpha<\frac{4-n/(n-1)}{n}$.
\end{description}
\end{rem}

\begin{rem}
It is not difficult to verify that the results of \cite{FurioliT},
\cite{Rogers} and \cite{WinT} (for $n=3$ and $\alpha<4/(n-2s)$) are
covered by Theorem~\ref{Bg1} and Theorem~\ref{Ls1}. The conclusions
of \cite{Katou} are also included in our results when
$\alpha>\frac{2s}{n-2s}$.
\end{rem}
The strategy of our proof is similar to the one used by Furioli and
Terraneo in \cite{FurioliT}, and it makes use of the negative order
homogeneous Besov space $\dot{B}^\sigma_{\rho,2}$ and Sobolev space
$\dot{H}^\sigma_\rho$ respectively.  For the choice  of $\rho$, in
addition to that used by Furioli and Terraneo in \cite{FurioliT}, we
can also select different indices. Generally speaking, if $u,v\in
L^\infty(I;\dot{H}^s)$ are two solutions of \eqref{NLS},  in order
that $u-v\in L^\infty(I;\dot{B}^\sigma_{\rho,2})$, the relationship
$s-\frac{n}{2}=\sigma-\frac{n}{\rho}$ is natural by the embedding
$\dot{H}^s\hookrightarrow\dot{B}^\sigma_{\rho,2}$ with $\sigma<0$.
However, the difference of two solutions sometimes has better
regularity in certain spaces than each of the solutions. We show
this better regularity for the subcritical case in the Part~3.2 and
for the critical case in the Part~4.3.1.

We also use  nonhomogeneous Strichartz estimates, which are
different from those used in \cite{FurioliT}. Furioli and Terraneo
applied the classical Strichartz estimates:
\begin{equation*}
\|\int^t_0e^{i(t-s)\Delta}f(s)\,ds\|_{L^{q_1}(I;L^{r_1})}\leq
C\|f\|_{L^{{q_2}'}(I;L^{{r_2}'})},
\end{equation*}
where $(e^{it\Delta }) _{ t\in \R }$ is the Schr\"o\-din\-ger group
and $f\in L^{{q_2}'}(I;L^{{r_2}'})$, and  the pairs $(q_i,r_i),\
i=1,2$ satisfy the admissibility conditions
$\frac{2}{q_i}=n(\frac{1}{2}-\frac{1}{r_i})$ and $2\leq r_i\leq
2n/(n-2)$ $(2\leq r_i\leq\infty \text{ if } n=1, \ 2\leq r_i<\infty
\text{ if } n=2)$. Therefore, in order to make the selected $\rho$
part of an admissible pair, the condition $s-1\leq\sigma\leq s$
should be satisfied. Furthermore, Furioli and Terraneo only settled
the case  $\alpha>1$, where the nonlinearity is locally Lipschitz
continuous. Their method does not apply to the case $\alpha<1$, when
the nonlinearity is not locally Lipschitz. We apply the general
Strichartz estimates, which are described in
Lemma~\ref{non-str-est}, $\rho$ being restricted to be part of a
``general" admissible pair. This improves the previous restriction
on $\sigma $. Our restrictions on $\sigma $, for instance given by
\eqref{sigma-cond1}-\eqref{sigma-cond2} or \eqref{sigma-cond} when
$\alpha\geq1$ and \eqref{sigma-cond1-1}-\eqref{sigma-cond2-2} or
\eqref{sigma-cond1-3}-\eqref{sigma-cond2-4} when $\alpha<1$,
corresponding to the different choices $\rho$. In fact, if
$s\geq\frac{1}{2},\ \alpha>\max\{1,\frac{2}{n-2s}\}$, in  light of
\eqref{sigma-cond1}-\eqref{sigma-cond2}, we may choose $\sigma=s-1$,
which is the choice made by Win and Tsutsumi in \cite{WinT}. For the
case $\alpha<1$, we use the fractional chain rule for a H\"older
continuous function (Lemma~\ref{FcrfaHcf}), then a result similar to
Lemma~2.3 in \cite{FurioliT} is obtained, which is applied to
control the nonlinearity. From the proof of the case $\alpha\geq1$,
we can see that the bound $(4s+4-\frac{n}{n-1})/(n-2s)$ comes from
the condition $\sigma+s\geq0$, which ensures \eqref{est3} to hold. A
similar argument can be used in the case $\alpha<1$.

 We also consider the
critical cases in the following results:
\begin{thm}\label{cri1}
Let $\alpha=\frac{n+2s}{n-2s}$ and $n=2$ with $0<s<1$ or $n=3$ with
$\frac{1}{4}<s<\frac{1}{2}$. Given $\varphi\in\dot{H}^s$ and $T>0$,
unconditional uniqueness holds in $L^\infty(0,T;\dot{H}^s)$ for
\eqref{NLS}.
\end{thm}
\begin{thm}\label{cri2}
Let $\alpha=\frac{4}{n-2s}$ and $n=3$ with $1/2<s<1$ or $n=4$ with
$1/3<s<1$ or $n\geq5$ with $s_0<s<1$, where $s_0$ is the smallest
solution of equation $4(n-1)s^2-(2n^2+8n-8)s+n^2=0$. Given
$\varphi\in\dot{H}^s$ and $T>0$, unconditional uniqueness holds in
$C([0,T];\dot{H}^s)$ for \eqref{NLS}.
\end{thm}
\begin{rem}
It follows from Theorem \ref{cri1} and \ref{cri2} that unconditional
uniqueness in the critical case is left open in the following cases:
\begin{description}
  \item [Case a] $n=2$, $\alpha=1$ and $s=0$;
  \item [Case b] $n=3$, $\alpha=\frac{n+2s}{n-2s}$ and $0\leq
  s\leq\frac{1}{4}$ or $s=\frac{1}{2}$;
  \item [Case c] $n=4$, $\alpha=\frac{4}{n-2s}$ and $0\leq
  s\leq\frac{1}{3}$;
  \item [Case d] $n\geq5$, $\alpha=\frac{4}{n-2s}$ and $0\leq s\leq
  s_0$.
\end{description}
\end{rem}
\begin{rem}
Note that Theorem~\ref{cri1} states uniqueness of solutions in
$L^\infty (0,T; \dot H^s)$, while Theorem~\ref{cri2} states
uniqueness for solutions in a stronger sense, i.e. solutions in
$C([0, T];\dot H^s)$. The fundamental reason is that, under the
assumptions of Theorem 1.5, when estimating the difference of two
solutions, there comes a factor of $T$ in the right-had side. So we
can choose $T$ sufficiently small so that the right hand side is
absorbed by the left hand side. However, under the assumptions of
Theorem \ref{cri2}, the coefficient is no longer dependent on time.
A similar difficulty appears in \cite{CLN} and \cite{WinT}. Using an
argument inspired by \cite{CLN,WinT}, we divide the nonlinearity by
high-low frequencies and use the norms
$L^\gamma_t\dot{H}^\sigma_\rho\cap L^a_t\dot{H}^\sigma_b$, where the
parameters $\sigma$, $\gamma$, $\rho$, $a$ and $b$ are chosen in
Section~4.
\end{rem}

\emph{\textbf{Notation:}}$H^s$ is the Sobolev space and $\dot{H}^s$
is the homogeneous Sobolev space, see Section~6.2 and Section~6.3 of
\cite{BerghL} respectively. Similarly, $B^s_{p,q}$ and
$\dot{B}^s_{p,q}$ are the Besov spaces and the homogeneous Besov
spaces, as defined in Section~6.2 and Section~6.3 of \cite{BerghL}.

The paper is organized as follows: in Section 2, we state and prove
some preparatory lemmas; in Section 3, we give the proofs of
Theorem~\ref{Bg1} and Theorem~\ref{Ls1}; Section 4 is devoted to the
proofs of Theorem~\ref{cri1} and Theorem~\ref{cri2}. Finally, we
present four figures at the end of the paper, displaying in
dimensions $n=3$, $n=4$, $n=5$ and $n\ge 6$, respectively, the
various regions where unconditional uniqueness is known or is still
an open problem.
\section{Preliminaries}
In this section, we present some  lemmas which we need. The first
one is nonhomogeneous Strichartz estimate which is due to Foschi
\cite{Foschi}. This estimate extends results of Strichartz
\cite{Strichartz}, Ginibre and Velo \cite{GINVEL}, Yajima
\cite{Yajima}, Cazenave and Weissler \cite{CAZWEI}, Keel and Tao
\cite{KeelT}.
\begin{definition}
 We say that the pair $(q,r)$ is $\frac{n}{2}$-acceptable if
\begin{equation}
1\leq q<\infty,\ 2\leq r\leq\infty,\
\frac{1}{q}<n(\frac{1}{2}-\frac{1}{r}),\ \text{or}\
(q,r)=(\infty,2).
\end{equation}
\end{definition}
\begin{lem}[Nonhomogeneous Strichartz estimate]\label{non-str-est}
Given any $\sigma\in\mathbb{R}$, the following properties holds:

Let $I$ be an interval of $\mathbb{R}$,  $J=\bar{I}$, and $0\in J$.
If $(q,r)$ is a $\frac{n}{2}$-acceptable pair and $f\in
L^{q'}(I;\dot{H}^{\sigma}_{r'})$, then for every
$\frac{n}{2}$-acceptable pair $(\gamma,\rho)$, there exists a
constant $C$ independent of $I$ such that
\begin{equation}\label{sob}
\|\int^t_0e^{i(t-s)\Delta}f(s)\,ds\|_{L^\gamma(I;\dot{H}^\sigma_\rho)}\leq
C\|f\|_{L^{q'}(I;\dot{H}^{\sigma}_{r'})},
\end{equation}
when $\gamma,\rho,q$ and $r$ verify the scaling condition
\begin{equation}\label{non-str-est-cond1}
\frac{1}{q}+\frac{1}{\gamma}=\frac{n}{2}(1-\frac{1}{r}-\frac{1}{\rho})
\end{equation}
and satisfy one of the following sets of conditions:
\begin{itemize}
  \item if $n=2$, we also require that $r,\rho<\infty$;
  \item if $n\geq3$, we distinguish two cases,
\end{itemize}

-non sharp case:
\begin{eqnarray}
&\frac{1}{q}+\frac{1}{\gamma}<1,\label{non-str-est-cond2}\\
&(\frac{n}{2}-1)\frac{1}{r}\leq\frac{n}{2\rho},\quad(\frac{n}{2}-1)\frac{1}{\rho}\leq\frac{n}{2r}\label{non-str-est-cond3};
\end{eqnarray}

-sharp case:
\begin{eqnarray}
&\frac{1}{q}+\frac{1}{\gamma}=1,\label{non-str-est-cond4}\\
&(\frac{n}{2}-1)\frac{1}{r}<\frac{n}{2\rho},\quad(\frac{n}{2}-1)\frac{1}{\rho}<\frac{n}{2r},\label{non-str-est-cond5}\\
&\frac{1}{r}\leq\frac{1}{q},\quad\frac{1}{\rho}\leq\frac{1}{\gamma}.&\label{non-str-est-cond6}
\end{eqnarray}

 The Sobolev space can be replaced by Besov space, where
the conditions $\gamma,q\geq2$ have to hold.
\end{lem}
\begin{proof}
The estimate without derivatives follows from \cite{Foschi}. The
proof for the Sobolev spaces is simple if we notice the fact
\begin{equation}
e^{i(t-s)\Delta}[\mathcal {F}^{-1}(|\xi|^\sigma \hat{f})]=\mathcal
{F}^{-1}[|\xi|^\sigma \mathcal {F}(e^{i(t-s)\Delta}f(s))],
\end{equation}
where $(e^{it\Delta }) _{ t\in \R }$ is the Schr\"o\-din\-ger group
and $\mathcal {F}$ is the Fourier transform.

 For the case of Besov
spaces, by the definition of the homogeneous Besov space (see
section~6.3 of \cite{BerghL}), we have
 \begin{multline}
\|\int^t_0e^{i(t-s)\Delta}f(s)\,ds\|_{L^\gamma(I;\dot{B}^\sigma_{\rho,2})}^2\\
=\Big\|\Big(\sum\limits_{j=-\infty}^\infty\big(2^{\sigma
j}\|\mathcal {F}^{-1}(\psi_j\mathcal
{F}(\int^t_0e^{i(t-s)\Delta}f(s)\,ds))\|_{L^\rho}\big)^2\Big)^\frac{1}{2}\Big\|_{L^\gamma(I)}^2\\
:=\|\Phi_j\|^2_{L^\gamma_Il_j^2L_x^\rho},
 \end{multline}
 where  $\mathcal {F}^{-1}\psi_j$ is the homogeneous dyadic decomposition.
By Minkowski's inequality and estimate of \cite{Foschi}, we have
\begin{multline}
\|\Phi_j\|^2_{L^\gamma_Il_j^2L_x^\rho}\leq\|\Phi_j\|^2_{l_j^2L^\gamma_IL_x^\rho}\\
\leq\|2^{\sigma j}\mathcal
 {F}^{-1}(\psi_j\hat{f}))\|_{l_j^2L^{q'}_IL_x^{r'}}^2
 \leq\|2^{\sigma
j}\mathcal
 {F}^{-1}(\psi_j\hat{f}))\|_{L^{q'}_Il_j^2L_x^{r'}}^2\\
 =\|f\|_{L^{q'}(I;\dot{B}^\sigma_{r',2})}^2
\end{multline} if
$\gamma, q\geq2$,
which completes the proof.
\end{proof}

For the Cauchy problem in $H^s$ spaces, we cannot avoid to estimate
the nonlinearity with some fractional derivative. Therefore, we need
the fractional chain rule and bilinear estimate for the nonlinearity
in Sobolev space and Besov space.
\begin{lem}[Product rule]\label{rule}
Let $s\in(0,1)$ and $1<r,p_1,p_2,q_1,q_2<\infty$ such that
$\frac{1}{r}=\frac{1}{p_i}+\frac{1}{q_i}$ for $i=1,2$. If $f\in
L^{p_1}\cap\dot{H}^s_{p_2}\cap\dot{B}^s_{p_2,2}$ and $g\in
L^{q_2}\cap\dot{H}^s_{q_1}\cap\dot{B}^s_{q_1,2}$, then
\begin{eqnarray}
\||\nabla|^s(fg)\|_{L^r}&\lesssim&\|f\|_{L^{p_1}}\||\nabla|^sg\|_{L^{q_1}}+\|g\|_{L^{q_2}}\||\nabla|^sf\|_{L^{p_2}},\label{sobolev rule}\\
\|fg\|_{\dot{B}^s_{r,2}}&\lesssim&\|f\|_{L^{p_1}}\|g\|_{\dot{B}^s_{q_1,2}}+\|g\|_{L^{q_2}}\|f\|_{\dot{B}^s_{p_2,2}}.\label{besov
rule}
\end{eqnarray}
\end{lem}
\begin{proof}
Estimate \eqref{sobolev rule} follows from Proposition~3.3 of
\cite{ChWe}. For the case of Besov space, using the equivalence of the norm (see theorem 6.3.1 in \cite{BerghL}), we have
\begin{equation}\label{equiv norm}
\|fg\|_{\dot{B}^s_{r,2}}=\Big(\int^\infty_0(t^{-s}\sup\limits_{|y|\leq
t}\|(fg)(\cdot-y)-(fg)(\cdot)\|_{L^r})^2\,\frac{dt}{t}\Big)^\frac{1}{2}.
\end{equation}
Note that
\begin{equation}
(fg)(\cdot-y)-(fg)(\cdot)=(f(\cdot-y)-f(\cdot))g(\cdot-y)+(g(\cdot-y)-g(\cdot))f(\cdot),
\end{equation}
then by H\"older inequality and \eqref{equiv norm}, we can show that
\eqref{besov rule} is true.
\end{proof}
\begin{lem}[\cite{ChWe}]\label{Lchainrule}
Suppose $G\in C^1(\mathbb{C}),\ s\in(0,1]$, and $1<p,p_1,p_2<\infty$
are such that $\frac{1}{p}=\frac{1}{p_1}+\frac{1}{p_2}$. Then,
\begin{equation}
\||\nabla|^sG(u)\|_{L^p}\lesssim\|G'(u)\|_{L^{p_1}}\||\nabla|^su\|_{L^{p_2}}.
\end{equation}
\end{lem}
\begin{lem}[Fractional chain rule for a H\"older continuous function,
Proposition A.1 in \cite{V07}]\label{FcrfaHcf}
 Let $G$ be a H\"older continuous
function of order $0<\alpha<1$. Then, for every $0<s<\alpha,\
1<p<\infty,$ and $\frac{s}{\alpha}<\sigma<1$, we have
\begin{equation}
\big\||\nabla|^sG(u)\big\|_{L^p}\lesssim\big\||u|^{\alpha-\frac{s}{\sigma}}\big\|_{L^{p_1}}\big\||\nabla|^\sigma
u\big\|_{L^{\frac{s}{\sigma}p_2}}^\frac{s}{\sigma},
\end{equation}
provided $\frac{1}{p}=\frac{1}{p_1}+\frac{1}{p_2}$ and
$(1-\frac{s}{\alpha\sigma})p_1>1$.
\end{lem}
\begin{lem}\label{bilinear est}
Let $-1<\sigma<0$ and $1<\rho,p_1,p_2,p_3,r<\infty$ such that
$\frac{1}{\rho'}=\frac{1}{p_1}+\frac{1}{p_2}=\frac{1}{p_3}+\frac{1}{r}$
and $\frac{1}{p_2}=\frac{1}{r}+\frac{\sigma}{n}$. Then for any $f\in
L^{p_3}\cap\dot{H}^{-\sigma}_{p_1}$ and $g\in\dot{H}^\sigma_\rho$,
we have
\begin{equation}\label{sob est}
\|fg\|_{\dot{H}^\sigma_{r'}}\lesssim\|g\|_{\dot{H}^\sigma_{\rho}}\big(\|f\|_{\dot{H}^{-\sigma}_{p_1}}+\|f\|_{L^{p_3}}\big).
\end{equation}
Furthermore, if $p_2\geq2$, then for any $f\in
L^{p_3}\cap\dot{B}^{-\sigma}_{p_1,2}$ and
$g\in\dot{B}^\sigma_{\rho,2}$, we have
\begin{equation}\label{bov est}
\|fg\|_{\dot{B}^\sigma_{r',2}}\lesssim\|g\|_{\dot{B}^\sigma_{\rho,2}}\big(\|f\|_{\dot{B}^{-\sigma}_{p_1,2}}+\|f\|_{L^{p_3}}\big).
\end{equation}
\end{lem}
\begin{proof}
We only prove the case of Sobolev spaces,  a similar argument can be
used for the case of  Besov spaces. By duality, to prove \eqref{sob
est}, we need only prove the following inequality
\begin{equation}
|<fg,h>|\lesssim\|g\|_{\dot{H}^\sigma_{\rho}}\big(\|f\|_{\dot{H}^{-\sigma}_{p_1}}+\|f\|_{L^{p_3}}\big)\|h\|_{\dot{H}^{-\sigma}_r},
\end{equation}
where $<\cdot,\cdot>$ denotes the $L^2$ scalar product.

By \eqref{sobolev rule} of Lemma~\ref{rule} , H\"older inequality
and Sobolev's embedding, it follows that
\begin{multline}
|<fg,h>|=|<g,\bar{f}h>|\leq\|g\|_{\dot{H}^\sigma_\rho}\|\bar{f}h\|_{\dot{H}^{-\sigma}_{\rho'}}\\
\lesssim\|g\|_{\dot{H}^\sigma_\rho}\big(\|f\|_{\dot{H}^{-\sigma}_{p_1}}\|h\|_{L^{p_2}}+\|f\|_{L^{p_3}}\|h\|_{\dot{H}^{-\sigma}_{r}}\big)\\
\lesssim\|g\|_{\dot{H}^\sigma_{\rho}}\big(\|f\|_{\dot{H}^{-\sigma}_{p_1}}+\|f\|_{L^{p_3}}\big)\|h\|_{\dot{H}^{-\sigma}_r}.
\end{multline}
\end{proof}

\section{The proof of Theorems~1.1 and ~1.2 }

In this section, we give the proofs of Theorem~\ref{Bg1} and
Theorem~\ref{Ls1}. We invoke some negative order Sobolev (or Besov)
spaces, the general nonhomogeneous Strichartz estimate and
properties of the solutions to achieve our goal. Let $u$ and $v$ be
two $L^\infty(0,T;\dot{H}^s)$ solutions of \eqref{NLS} with the same
initial data $\varphi$ and $T>0$. The Parts~3.1 and 3.2 are devoted
to the proof of the case $1\leq\alpha$, and the rest illustrate the
proof for $0<\alpha<1$. For the sake of simplicity, we denote
$f(u)=c|u|^\alpha u$.

\subsection{Usual regularity property case}

We consider the space $L^\gamma(0,T;\dot{B}^\sigma_{\rho,2})$ for
certain $n/2$-acceptable pair$(\gamma,\rho)$, with
$\frac{1}{\rho}=\frac{\sigma}{n}+\frac{1}{2}-\frac{s}{n}$, $\sigma
<0$, where $\sigma$ and $\gamma$ can be fixed later. We say $u-v$
has usual regularity property if it belongs to the same auxiliary
space $L^\gamma(0,T;\dot{B}^\sigma_{\rho,2})$ as that $u,v$ belong
to, by embedding $\dot{H}^s\hookrightarrow\dot{B}^\sigma_{\rho,2}$
for $u,v \in L^\infty(0,T;\dot{H}^s)$ with finite time $T$.
 Our aim is to show the uniqueness in
the space $L^\gamma(0,T;\dot{B}^\sigma_{\rho,2})$.

By using  Duhamel's formula and  Lemma~\ref{non-str-est} in non
sharp case, we have
\begin{equation}\label{est}
\|u-v\|_{L^\gamma(0,T;\dot{B}^\sigma_{\rho,2})}\lesssim\|f(u)-f(v)\|_{L^{q'}(0,T;\dot{B}^\sigma_{r',2})},
\end{equation}
where
$\frac{1}{r}=\frac{1}{2}-\frac{\sigma}{n}+\frac{s}{n}-\frac{n-2s}{2n}\alpha,$
$(q,r)$ is a $\frac{n}{2}$-acceptable pair and $\gamma,\rho,q$ and $r$
satisfy the conditions
\eqref{non-str-est-cond1}-\eqref{non-str-est-cond3} with
$\gamma,q\geq2$.

Given $u, v\in \C$, we have
\begin{multline*}
f(u)- f(v)= (u-v) \int _0^1 \partial _z f (v+\theta (u-v)) \,d\theta + \\
( \overline{u-v} ) \int _0^1 \partial _{ \overline{z} }f (v+\theta
(u-v)) \,d\theta,
\end{multline*}
or, in short,
\begin{equation} \label{fLemLipd}
f(u)- f(v)= (u-v) \int _0^1 f' (v+\theta (u-v)) \,d\theta.
\end{equation}

If
$\frac{1}{p_1}=\frac{n-2s}{2n}\alpha-\frac{\sigma}{n},$
$\frac{1}{p_3}=\frac{n-2s}{2n}\alpha$,  $\sigma$ and $ r$ satisfy
the conditions of Lemma~\ref{bilinear est}, then we have
\begin{multline}\label{est1}
\|f(u)-f(v)\|_{\dot{B}^\sigma_{r',2}}\\
\lesssim\|u-v\|_{\dot{B}^\sigma_{\rho,2}}\Big(\int^1_0\|f'
(v+\theta (u-v))\|_{\dot{B}^{-\sigma}_{p_1,2}}\,d\theta+\int^1_0\|f'
(v+\theta (u-v))\|_{L^{p_3}}\,d\theta\Big).
\end{multline}

By the form of $f'$ and Sobolev embedding
$\dot{H}^s\hookrightarrow
L^\frac{2n}{n-2s}$, we see
\begin{equation}\label{est2}
\int^1_0\|f' (v+\theta
(u-v))\|_{L^{p_3}}\,d\theta\lesssim(\|u\|_{L^{p_3\alpha}}+\|v\|_{L^{p_3\alpha}})^\alpha\lesssim(\|u\|_{\dot{H}^s}+\|v\|_{\dot{H}^s})^\alpha.
\end{equation}

Using the equivalent norm of Besov space(see theorem 6.3.1 in
\cite{BerghL}), H\"older's inequality, the embedding
$\dot{H^s}\hookrightarrow L^\frac{2n}{n-2s}$,
$\dot{H^s}\hookrightarrow \dot{B}^{-\sigma}_{l,2}$, and
$\alpha\geq1$, we have
\begin{multline}\label{est3}
\int^1_0\|f' (v+\theta
(u-v))\|_{\dot{B}^{-\sigma}_{p_1,2}}\,d\theta\\
=\int^1_0\Big\{\int^\infty_0t^{2\sigma}(\sup\limits_{|y|\leq t}\|(f'
(v+\theta (u-v)))_y-(f' (v+\theta
(u-v)))\|_{L^{p_1}})^2\,\frac{dt}{t}\Big\}^\frac{1}{2}\,d\theta\\
\lesssim\int^1_0\int^1_0\Big\{\int^\infty_0t^{2\sigma}\Big[\sup\limits_{|y|\leq
t}\|\eta(\theta u+(1-\theta)v)_y+(1-\eta)(\theta
u+(1-\theta)v)\|_{L^\frac{2n}{n-2s}}^{\alpha-1}\\
\|(\theta
u+(1-\theta)v)_y-(\theta
u+(1-\theta)v)\|_{L^l}\Big]^2\,\frac{dt}{t}\Big\}^\frac{1}{2}\,d\eta\,d\theta\\
\lesssim(\|u\|_{L^\frac{2n}{n-2s}}+\|u\|_{L^\frac{2n}{n-2s}})^{\alpha-1}(\|u\|_{\dot{B}^{-\sigma}_{l,2}}+\|v\|_{\dot{B}^{-\sigma}_{l,2}})\\
\lesssim(\|u\|_{\dot{H}^s}+\|u\|_{\dot{H}^s})^\alpha,
\end{multline}
where $u(\cdot-y):=u_y$,
$\frac{1}{l}=\frac{1}{2}-\frac{s}{n}-\frac{\sigma}{n}$ and
$s\geq-\sigma$.

Then, by \eqref{est1}, \eqref{est2}, \eqref{est3} and H\"older's
inequality in time, it follows from \eqref{est} that
\begin{equation}\label{LabsorbR}
\|u-v\|_{L^\gamma(0,T;\dot{B}^\sigma_{\rho,2})}\lesssim
T^{1-\frac{1}{q}-\frac{1}{\gamma}}\Big(\|u\|_{L^\infty(0,T;\dot{H}^s)}+\|v\|_{L^\infty(0,T;\dot{H}^s)}\Big)^\alpha\|u-v\|_{L^\gamma(0,T;\dot{B}^\sigma_{\rho,2})}.
\end{equation}
Therefore, if $T$ is sufficiently small, we have $u=v$ on $[0,T]$.

We summarize the conditions that we have imposed so far on the
parameters $q,r,\gamma,\rho,\sigma$:
\begin{enumerate}
  \item\label{cond1}  the choices of $\rho$ and $r$:
  $\frac{1}{\rho}=\frac{\sigma}{n}+\frac{1}{2}-\frac{s}{n},$
  $\frac{1}{r}=\frac{1}{2}-\frac{\sigma}{n}+\frac{s}{n}-\frac{n-2s}{2n}\alpha$;
  \item\label{cond2}  $(\gamma,\rho),\ (q,r)$ being $\frac{n}{2}$ -acceptable
  pairs;
  \item\label{cond3} $(\gamma,\rho),\ (q,r)$ satisfying the conditions
  \eqref{non-str-est-cond1}-\eqref{non-str-est-cond3} and $\gamma,q\geq2$;
  \item\label{cond4} conditions on $\sigma$ and $r$ for the validity of
  Lemma~\ref{bilinear est},
  \[
  -1<\sigma<0,\ 0<\frac{1}{r}+\frac{\sigma}{n}\leq\frac{1}{2},
  \]
  where the second is equivalent to
  $\frac{2s}{n-2s}\leq\alpha<\frac{n+2s}{n-2s}$;
  \item\label{cond5} condition on $\sigma$ for the validity of \eqref{est3},
  \[
s\geq-\sigma.
  \]
\end{enumerate}
From the conditions
\eqref{cond1}-\eqref{cond5},  the restrictions on $\sigma$ are
finally derived:
\begin{eqnarray}
&\max\{ -s,
s-\frac{n}{2(n-1)}-\frac{(n-2)(n-2s)}{4(n-1)}\alpha\}\leq\sigma\leq s+\frac{n}{2(n-1)}-\frac{n(n-2s)}{4(n-1)}\alpha,\label{sigma-cond1}\\
&\max\{s-\frac{1}{2}-\frac{n-2s}{4}\alpha,
s-\frac{n-2s}{2}\alpha,s-\frac{n}{2}\}<\sigma<\min\{0,
s+\frac{1}{2}-\frac{n-2s}{4}\alpha\}.\label{sigma-cond2}
\end{eqnarray}

It follows from a simple calculation that the set consisting of the
elements which satisfy the conditions
\eqref{sigma-cond1}-\eqref{sigma-cond2} is non-empty if $s, \alpha$
and $n$ satisfy the conditions
\begin{equation}\label{cond-a1}
\left\{\begin{array}{ll}
       1\leq\alpha<\min{\{\frac{4}{n-2s},\frac{n+2s}{n-2s}\}},\\
       \frac{2s}{n-2s}<\alpha\leq(2+8s(1-\frac{1}{n}))/(n-2s),\\
       0<s<1 \text{ and }n=3,4,5.
       \end{array}
       \right.
\end{equation}
Therefore, under the conditions of \eqref{cond-a1}, we establish the
unconditional uniqueness of \eqref{NLS}.

\subsection{Better regularity property case}
This part is devoted to the study of
the cases
\begin{equation}\label{cd}
\left\{\begin{array}{ll}
       \max\{1,(2+8s(1-\frac{1}{n}))/(n-2s)\}<\alpha<\min{\{\frac{4}{n-2s},\frac{n+2s}{n-2s}\}},\\
       0<s<1 \text{ and }n=3,4,5.
       \end{array}
       \right.
\end{equation}

 We consider the space
$L^\gamma(0,T;\dot{B}^\sigma_{\rho,2})$ for the $n/2$-acceptable
pair $(\gamma,\rho)$, with
$\frac{1}{\rho}=\frac{\sigma}{n}+\frac{1}{2}-\frac{s}{n}+\frac{n-2s}{2n}\alpha-\frac{2}{n}$
and $\sigma < 0$. It is clear that $u, v$ are no longer in
$L^\gamma(0,T;\dot{B}^\sigma_{\rho,2})$ because of
$\dot{H}^s\nsubseteq\dot{B}^\sigma_{\rho,2}$. However, in some
restricted conditions on $\rho$, we can show $u-v \in
L^\gamma(0,T;\dot{B}^\sigma_{\rho,2})$.

Since $\sigma<0$,  we have the embedding
$L^\frac{2n}{(n-2s)(\alpha+1)}\hookrightarrow\dot{B}^\sigma_{p',2}$
(or
$L^\frac{2n}{(n-2s)(\alpha+1)}\hookrightarrow\dot{H}^\sigma_{p'}$),
where $p'=\frac{2n}{2\sigma+(n-2s)(\alpha+1)}$, and
then applying H\"older's inequality
and Sobolev embedding $\dot{H}^s\hookrightarrow L^\frac{2n}{n-2s}$,
one can get
\begin{multline}\label{bd}
\|f(u)-f(v)\|_{\dot{B}^\sigma_{p',2}}
\lesssim\|f(u)-f(v)\|_{L^\frac{2n}{(n-2s)(\alpha+1)}}\\
\lesssim\|(|u|^\alpha+|v|^\alpha)|u-v|\|_{L^\frac{2n}{(n-2s)(\alpha+1)}}
\lesssim\big(\|u\|_{\dot{H}^s}^\alpha+\|v\|_{\dot{H}^s}^\alpha\big)\|u-v\|_{\dot{H}^s}.
\end{multline}

Let $(\lambda,p)$ be an $\frac{n}{2}-$acceptable pair, and it is
easy to verify that if
$\frac{\frac{n}{2}+\frac{2}{n}-2}{n-1}<\frac{1}{\rho}<\frac{n-2}{2(n-1)}$,
then we can choose $\gamma$ and $\lambda$ such that $(\gamma,\rho)$
and $(\lambda,p)$ satisfy the conditions \eqref{non-str-est-cond1}
and \eqref{non-str-est-cond4}-\eqref{non-str-est-cond6}. Then by the
sharp case of Lemma~2.1, finite time $T$ and \eqref{bd}, it follows
that
\begin{equation}
\|u-v\|_{L^\gamma(0,T;\dot{B}^\sigma_{\rho,2})}\lesssim\|f(u)-f(v)\|_{L^\lambda(0,T;\dot{B}^\sigma_{p',2})}<+\infty.
\end{equation}
Therefore, for any
$0<\alpha<\min\{\frac{4}{n-2s},\frac{n+2s}{n-2s}\}$, we have that
$u-v$ belongs to the space $L^\gamma(0,T;\dot{B}^\sigma_{\rho,2})$
(or $L^\gamma(0,T;\dot{H}^\sigma_{\rho})$).

The rest is the same as what we did in Part~3.1, except
the selection of $\rho$:
$\frac{1}{\rho}=\frac{\sigma}{n}+\frac{1}{2}-\frac{s}{n}+\frac{n-2s}{2n}\alpha-\frac{2}{n}$.
Then after a series of calculations, we can show that $\sigma$ has
to satisfy the conditions
\begin{equation}\label{sigma-cond}
\max\{-s,s+\frac{3n-4}{2n-2}-\frac{3n-4}{4n-4}(n-2s)\alpha\}\leq\sigma<\min\{0,s+\frac{3n-4}{2n-2}-\frac{n-2s}{2}\alpha\}.
\end{equation}

Obviously,  $\sigma$ is existent when $\alpha\in
[\frac{2+[(n-1)4s/(3n-4)]}{n-2s}, \frac{4s+4-n/(n-1)}{n-2s})$.
In view of \eqref{cd}, we have
\[
\frac{2+[(n-1)4s/(3n-4)]}{n-2s}<(2+8s(1-\frac{1}{n}))/(n-2s)<\frac{4s+4-n/(n-1)}{n-2s}.
\]
 Therefore,  unconditional uniqueness  is proved under
the assumptions
\[
\max\{(2+8s(1-\frac{1}{n}))/(n-2s),1\}<\alpha<\min\{\frac{4s+4-n/(n-1)}{n-2s},\frac{4}{n-2s},\frac{n+2s}{n-2s}\}.
\]

In summary, by the results of Part~3.1 and Part~3.2, if $\alpha,\ s$
and $n$ satisfy the conditions
\begin{equation}
\left\{\begin{array}{ll}
       \max\{1,\frac{2s}{n-2s}\}\leq\alpha<\min{\{\frac{4}{n-2s},\frac{n+2s}{n-2s},\frac{4s+4-n/(n-1)}{n-2s}\}},\\
       0<s<1 \text{ and }n=3,4,5,
       \end{array}
       \right.
\end{equation}
we have  unconditional uniqueness of \eqref{NLS}, which conclude
the proof of Theorem~\ref{Bg1}.

\subsection{The proof of theorem~\ref{Ls1}}
In this subsection, we give a sketch of proof of Theorem~\ref{Ls1}.
The proof proceeds follows  that of Theorem~\ref{Bg1}. Instead of
$L^\gamma(0,T;\dot{B}^\sigma_{\rho,2})$ there, we use the space
$L^\gamma(0,T;\dot{H}^\sigma_\rho)$ for  $\alpha<1$, and
Lemma~\ref{bilinear est} in Sobolev version. The main difference is
how to get a similar estimate as \eqref{est3}, now we should
consider $\int^1_0\|f' (v+\theta
(u-v))\|_{\dot{H}^{-\sigma}_{p_1}}\,d\theta$.

Since $f'\in C^{0,\alpha}$, Lemma~\ref{FcrfaHcf} and Sobolev
embedding $\dot{H}^s\hookrightarrow L^\frac{2n}{n-2s}$ lead to
\begin{multline}
\int^1_0\|f' (v+\theta
(u-v))\|_{\dot{H}^{-\sigma}_{p_1}}\,d\theta\lesssim\int^1_0\|v+\theta(u-v)\|_{L^\frac{2n}{n-2s}}^{\alpha+\frac{\sigma}{s}}\|v+\theta(u-v)\|_{\dot{H}^s}^{\frac{-\sigma}{s}}\,d\theta\\
\lesssim(\|u\|_{\dot{H}^s}+\|u\|_{\dot{H}^s})^\alpha.
\end{multline}

Similarly, when $\rho$ is chosen as
$\frac{1}{\rho}=\frac{\sigma}{n}+\frac{1}{2}-\frac{s}{n}$, we can
summarize the conditions imposed on  parameters $q,r,\gamma,\rho,$
and $\sigma$:
\begin{enumerate}
  \item\label{cond1-1}  the choices of $\rho$ and $r$:
  $\frac{1}{\rho}=\frac{\sigma}{n}+\frac{1}{2}-\frac{s}{n},$
  $\frac{1}{r}=\frac{1}{2}-\frac{\sigma}{n}+\frac{s}{n}-\frac{n-2s}{2n}\alpha$;
  \item\label{cond2-2}  $(\gamma,\rho),\ (q,r)$ being $\frac{n}{2}$ -acceptable
  pairs;
  \item\label{cond3-3} $(\gamma,\rho),\ (q,r)$ satisfying the conditions
  \eqref{non-str-est-cond1}-\eqref{non-str-est-cond3};
  \item\label{cond4-4} conditions on $\sigma$ and $r$ for the validity of
  Lemma~\ref{bilinear est},
  \[
  -1<\sigma<0,\ 0<\frac{1}{r}+\frac{\sigma}{n}\leq1,
  \]
  where the second one is equivalent to
  $\frac{2s-n}{n-2s}\leq\alpha<\frac{n+2s}{n-2s}$;
  \item\label{cond5-5} condition on $\sigma$ for the validity of Lemma~\ref{FcrfaHcf},
  \[
   -\alpha s<\sigma.
  \]
\end{enumerate}

These conditions still infer the conditions on $\sigma$,
\begin{eqnarray}
&\max\{s-\frac{n-2s}{2}\alpha, s-\frac{n}{2}, -\alpha s\}<\sigma<0,\label{sigma-cond1-1}\\
&s-\frac{n}{2(n-1)}-\frac{(n-2)(n-2s)}{4(n-1)}\alpha\leq\sigma\leq
s+\frac{n}{2(n-1)}-\frac{n(n-2s)}{4(n-1)}\alpha.\label{sigma-cond2-2}
\end{eqnarray}

Therefore, in order to establish
unconditional uniqueness of \eqref{NLS}, we have to find
$q,r,\gamma,\rho$ and $\sigma$ fulfill all of the restrictions.
Through a series of calculations, these parameters can be chosen if
$s,\alpha$ and $n$ satisfy one of the following conditions:

\begin{itemize}
  \item when $n=3$,
  \begin{equation}\left\{\begin{array}{ll}
                               \frac{2s}{n-2s}<\alpha<1,\ \ &\text{if
                               }s\geq\frac{9}{14},\\
                               \frac{2s}{n-2s}<\alpha<\min{\{1,\frac{2+4s(1-\frac{1}{n})}{(n-2s)-4s(1-\frac{1}{n})}\}},&\text{if
                               }s\leq\frac{9}{14},\\
                               \end{array}
                               \right.
        \end{equation}
  \item when $n=4$,
  \begin{equation}\left\{\begin{array}{ll}
                               \frac{2s}{n-2s}<\alpha<1,\ \ &\text{if
                               }s\geq\frac{4}{5},\\
                               \frac{2s}{n-2s}<\alpha<\min{\{1,\frac{2+4s(1-\frac{1}{n})}{(n-2s)-4s(1-\frac{1}{n})}\}},&\text{if
                               }s\leq\frac{4}{5},\\
                               \end{array}
                               \right.
           \end{equation}
  \item when $n=5$,
  \begin{equation}\left\{\begin{array}{ll}
                               \frac{2s}{n-2s}<\alpha<1,\ \ &\text{if
                               }s\geq\frac{25}{26},\\
                               \frac{2s}{n-2s}<\alpha<\min{\{1,\frac{4}{n-2s},\frac{2+4s(1-\frac{1}{n})}{(n-2s)-4s(1-\frac{1}{n})}\}},&\text{if
                               }s\leq\frac{25}{26},\\
                               \end{array}
                               \right.
          \end{equation}
  \item  when $n\geq6$,
  \begin{equation}
  \frac{2s}{n-2s}<\alpha<\min{\{\frac{4}{n-2s},\frac{2+4s(1-\frac{1}{n})}{(n-2s)-4s(1-\frac{1}{n})}\}}.
   \end{equation}
\end{itemize}

Next, we consider the case
$\frac{2+4s(1-\frac{1}{n})}{(n-2s)-4s(1-\frac{1}{n})}\leq\alpha<\min\{1,\frac{4}{n-2s},\frac{n+2s}{n-2s}\}$.
As showed in Part~3.2 ,  $u-v$ is in the space
$L^\gamma(0,T;\dot{H}^\sigma_\rho)$ with
$\frac{1}{\rho}=\frac{\sigma}{n}+\frac{1}{2}-\frac{s}{n}+\frac{n-2s}{2n}\alpha-\frac{2}{n}$
. Therefore, we invoke the property and the same argument as above
to reduce the restrictions of $\sigma$:
\begin{eqnarray}
&\max\{s+2-(n-2s)\alpha, -\alpha s\}<\sigma<s+2-\frac{n-2s}{2}\alpha-\frac{n}{2(n-1)},\label{sigma-cond1-3}\\
&s+2-\frac{n}{2(n-1)}-\frac{3n/4-1}{n-1}(n-2s)\alpha<\sigma<0.\label{sigma-cond2-4}
\end{eqnarray}

In order to ask that $\sigma$ satisfies the conditions
\eqref{sigma-cond1-3}-\eqref{sigma-cond2-4}, the following
relationship has to be satisfied
\[
\frac{2ns+3n-2s-4}{(3n/2-2)(n-2s)}<\alpha<\frac{2s+4-\frac{n}{n-1}}{n-4s}.
\]

Noticed the prior assumption
$\frac{2+4s(1-\frac{1}{n})}{(n-2s)-4s(1-\frac{1}{n})}\leq\alpha<\min\{1,\frac{4}{n-2s},\frac{n+2s}{n-2s}\}$,
we can see
\[
\frac{2ns+3n-2s-4}{(3n/2-2)(n-2s)}<\frac{2+4s(1-\frac{1}{n})}{(n-2s)-4s(1-\frac{1}{n})}<\frac{2s+4-\frac{n}{n-1}}{n-4s}.
\]
Therefore, we have shown unconditional uniqueness under the
following conditions:
\[
\frac{2+4s(1-\frac{1}{n})}{(n-2s)-4s(1-\frac{1}{n})}\leq\alpha<\min\{1,\frac{4}{n-2s},\frac{n+2s}{n-2s},\frac{2s+4-\frac{n}{n-1}}{n-4s}\}.
\]

In summary, by the results of above, we have unconditional
uniqueness of \eqref{NLS} if $\alpha,s$ and $n$ satisfy the
following conditions:
\begin{itemize}
  \item when $n=3$,
  \begin{equation}\left\{\begin{array}{ll}
                               \frac{2s}{3-2s}<\alpha<1,\ \ &\text{if
                               }s>\frac{3}{4},\\
                               \frac{2s}{3-2s}<\alpha<\min{\{1,\frac{2s+\frac{5}{2}}{3-4s}\}},&\text{if
                               }s\leq\frac{3}{4},\\
                               \end{array}
                               \right.
        \end{equation}
  \item  when $n\geq4$,
  \begin{equation}
  \frac{2s}{n-2s}<\alpha<\min{\{1,\frac{4}{n-2s},\frac{2s+4-\frac{n}{n-1}}{n-4s}\}}.
   \end{equation}
\end{itemize} Hence, we finish the proof of Theorem~\ref{Ls1}.

\section{The proof of Theorems~1.5 and ~1.6 }

In this section, we give the proof of Theorems~1.5 and ~1.6. By the
argument of the subcritical case, one can find $\sigma$ is a
function with respect to parameters $\alpha$, $s$ and $n$. If we
consider the critical case, $\alpha$ can be determined by $s$. So
for some special dimensions, we can fix the choice of $\sigma$.
\subsection{The case of $n=2$, $\alpha=\frac{2+2s}{2-2s}$ and $0<s<1$}

\subsubsection{\textbf{The case of $n=2$, $\alpha=\frac{2+2s}{2-2s}$ and
$0<s<\frac{1}{2}$}}

In this situation, we select $\sigma=-s+\varepsilon$,
$\frac{1}{\lambda}=\frac{1}{2}+\frac{\varepsilon}{2}$,
$(\frac{1}{a},\frac{1}{b})=(\frac{s}{2},\frac{1}{2}-s)$, where
$\varepsilon$ is a sufficiently small constant such that
$0<\varepsilon<s$.

By Sobolev Embedding $L^1\hookrightarrow
\dot{H}^\sigma_{\frac{2}{2+\sigma}}$, $\dot{H}^s\hookrightarrow
L^\frac{4}{2-2s}$, H\"older inequality and
$|f(u)-f(v)|\lesssim(|u|^\alpha+|v|^\alpha)(u-v)$, we can see
\begin{multline*}
\|f(u)-f(v)\|_{\dot{H}^\sigma_{\frac{2}{2+\sigma}}}\lesssim\|f(u)-f(v)\|_{L^1}\\
\lesssim(\|u\|_{L^\frac{4}{2-2s}}^\alpha+\|v\|_{L^\frac{4}{2-2s}}^\alpha)\|u-v\|_{L^\frac{4}{2-2s}}\lesssim\|u,v\|_{\dot{H}^s}^{\alpha+1}.
\end{multline*}
It is easy to see that $(a,b)$ and $(\lambda,-\frac{2}{\sigma})$ are
$\frac{n}{2}$-acceptable pairs, then by Lemma~\ref{non-str-est}, we
can get
\begin{equation}\label{bd1}
\|u-v\|_{L^a(0,T;\dot{H}^\sigma_b)}\lesssim\|f(u)-f(v)\|_{L^{\lambda'}(0,T;\dot{H}^\sigma_\frac{2}{2+\sigma})}\lesssim
T^\frac{1}{\lambda'}\|u,v\|_{L^\infty(0,T;\dot{H}^s)}^{\alpha+1},
\end{equation}
where $(a,b)$ and $(\lambda,-\frac{2}{\sigma})$ satisfy the
conditions of $n=2$ in Lemma~\ref{non-str-est}. From \eqref{bd1},
one can find if $u,v\in L^\infty(0,T;\dot{H}^s)$, then $u-v\in
L^a(0,T;\dot{H}^\sigma_b).$

Furthermore, we choose $(q,r)=(2,\frac{2}{s})$ which is an
$\frac{n}{2}$-acceptable pair. By a simple calculation, one can see
$(a,b)$ and $(q,r)$ satisfying the conditions of
Lemma~\ref{non-str-est}, then it follows that
\begin{equation}\label{n2iq2}
\|u-v\|_{L^a(0,T;\dot{H}^\sigma_b)}\lesssim\|f(u)-f(v)\|_{L^{q'}(0,T;\dot{H}^\sigma_{r'})}.
\end{equation}
 If we select $\frac{1}{p_1}=\frac{1}{2}+s-\frac{\varepsilon}{2}$
 and $\frac{1}{p_3}=\frac{1+s}{2}$, then $r,b,p_1$ and $p_3$ satisfy the conditions of bilinear estimate
 Lemma~\ref{bilinear est}. So it follows that
 \begin{multline}\label{n2iq3}
\|f(u)-f(v)\|_{\dot{H}^\sigma_{r'}}\\
\lesssim\|u-v\|_{\dot{H}^\sigma_b}\Big(\int^1_0\|f'(\theta
u+(1-\theta)v)\|_{\dot{H}^{-\sigma}_{p_1}}\,d\theta+\int^1_0\|f'(\theta
u+(1-\theta)v)\|_{L^{p_3}}\,d\theta\Big).
\end{multline}
For the second term of the right hand side of \eqref{n2iq3}, by
Sobolev Embedding $\dot{H}^s\hookrightarrow L^\frac{4}{2-2s}$, we
can obtain
\begin{equation}\label{n2iq4}
\int^1_0\|f'(\theta
u+(1-\theta)v)\|_{L^{p_3}}\,d\theta\lesssim(\|u\|_{L^{p_3\alpha}}+\|u\|_{L^{p_3\alpha}})^\alpha\lesssim(\|u\|_{\dot{H}^s}+\|u\|_{\dot{H}^s})^\alpha.
\end{equation}
For the first term of the right hand side of \eqref{n2iq3}, by
Lemma~\ref{Lchainrule} and $\dot{H}^s\hookrightarrow
\dot{H}^{-\sigma}_l$, where
$\frac{1}{l}=\frac{1}{2}-\frac{\varepsilon}{2}$, we can obtain
\begin{multline}\label{n2iq5}
\int^1_0\|f'(\theta
u+(1-\theta)v)\|_{\dot{H}^{-\sigma}_{p_1}}\,d\theta\lesssim\Big(\|u\|_{L^\frac{4}{2-2s}}+\|v\|_{L^\frac{4}{2-2s}}\Big)^{\alpha-1}(\|u\|_{\dot{H}^{-\sigma}_l}+\|v\|_{\dot{H}^{-\sigma}_l})\\
\lesssim(\|u\|_{\dot{H}^s}+\|v\|_{\dot{H}^s})^\alpha.
\end{multline}
In conclusion, by \eqref{n2iq2}-\eqref{n2iq5} and H\"older
inequality on time, we can obtain
\begin{equation}
\|u-v\|_{L^a(0,T;\dot{H}^\sigma_b)}\lesssim
T^{\frac{1}{2}-\frac{1}{a}}\Big(\|u\|_{L^\infty(0,T;\dot{H}^s)}+\|v\|_{L^\infty(0,T;\dot{H}^s)}\Big)^\alpha\|u-v\|_{L^a(0,T;\dot{H}^\sigma_b)},
\end{equation}
which shows the unconditional uniqueness if $T$ sufficiently small.

\subsubsection{\textbf{The case of $n=2$, $\alpha=\frac{2+2s}{2-2s}$ and $\frac{1}{2}\leq
s<1$}} In this case, the conclusion follows from the same argument as above, but the choice of
\begin{eqnarray*}
&\sigma=s-1+2\varepsilon,\
&\frac{1}{a}=\frac{1}{2}-\frac{\varepsilon}{2},\quad
\frac{1}{b}=\frac{\varepsilon}{2},\\
&\frac{1}{\lambda}=\frac{s}{2}+\varepsilon,\
&\frac{1}{q}=\frac{s}{2}+\frac{\varepsilon}{2},\quad
\frac{1}{r}=\frac{1}{2}-\frac{\varepsilon}{2}-\frac{s}{2},\\
&\frac{1}{p_1}=1-\varepsilon,\
&\frac{1}{p_3}=\frac{1}{2}+\frac{s}{2},
\end{eqnarray*}
where $\varepsilon$ is a sufficiently small constant such that
$0<\varepsilon<\frac{1}{2}-\frac{s}{2}$.
\subsection{The case of $n=3$, $\alpha=\min\{\frac{3+2s}{3-2s},\frac{4}{3-2s}\}$}
\subsubsection{\textbf{The case of $n=3$, $\alpha=\frac{3+2s}{3-2s}$ and $\frac{1}{4}<s<\frac{1}{2}$}}

Similar to the case of $n=2$,  in this case, to get the conclusion we need to use the non-sharp case of
Lemma~\ref{non-str-est} for $n=3$ and choose
\begin{eqnarray*}
&\sigma=-s,\ &\frac{1}{a}=\frac{1}{2}-\frac{1}{b},\\
&\frac{1}{\lambda}=1-\frac{s}{2}-\frac{1}{2b},\
&\frac{1}{q}=\frac{1}{4}+\frac{s}{2}+\frac{1}{b},\quad
\frac{1}{r}=\frac{1}{2}-\frac{1}{b}-\frac{s}{3},\\
&\frac{1}{p_1}=\frac{1}{2}+\frac{2s}{3},\
&\frac{1}{p_3}=\frac{1}{2}+\frac{s}{3},
\end{eqnarray*}
and $\frac{1}{b}$ satisfies
$\frac{1}{3}-\frac{s}{3}<\frac{1}{b}<\min\{s,\frac{1}{2}-\frac{2s}{3}\}$.

So far, we have completed the proof of Theorem~1.5.

\subsubsection{\textbf{The case of $n=3$, $\alpha=\frac{4}{3-2s}$ and $\frac{1}{2}<s<1$}}

In this case, we choose $\sigma=s-1$. By Lemma~\ref{non-str-est} (or
classical Strichartz estimates), we can obtain
\begin{equation}\label{n3iq}
\|u-v\|_{L^2(0,T;\dot{H}^\sigma_6)}+\|u-v\|_{L^4(0,T;\dot{H}^\sigma_3)}\lesssim\|f(u)-f(v)\|_{L^2(0,T;\dot{H}^\sigma_{6/5})}.
\end{equation}
By Sobolev Embedding
$L^\frac{6}{7-2s}\hookrightarrow\dot{H}^\sigma_{6/5}$,
$\dot{H}^s\hookrightarrow L^\frac{6}{3-2s}$ and
$|f(u)-f(v)|\lesssim(|u|^\alpha+|v|^\alpha)(u-v)$, we have
\begin{equation*}
\|f(u)-f(v)\|_{\dot{H}^\sigma_{6/5}}\lesssim\|u,v\|_{\dot{H}^s}^{\alpha+1},
\end{equation*}
it follows $u-v\in L^2(0,T;\dot{H}^\sigma_6)\cap
L^4(0,T;\dot{H}^\sigma_3)$.

We denote that
\begin{multline}
f(u)-f(v)=\Big[\int^1_0f'(\theta
u+(1-\theta)v)\,d\theta\Big](u-v)\\
=\Big\{\int^1_0[P_{\leq N}f'(\theta
u+(1-\theta)v)]\,d\theta\Big\}(u-v)+\Big\{\int^1_0[P_{>N}f'(\theta
u+(1-\theta)v)]\,d\theta\Big\}(u-v)\\
:=(I)+(II).
\end{multline}

For the term (II), by using Lemma~\ref{bilinear est} with
$\frac{1}{p_1}=1-\frac{s}{3}$ and $\frac{1}{p_3}=\frac{2}{3}$, we
can obtain
\begin{multline}\label{n3iq3}
\|(II)\|_{L^2(0,T;\dot{H}^\sigma_{6/5})}\lesssim\Big\{\|\int^1_0[P_{>N}f'(\theta
u+(1-\theta)v)]\,d\theta\|_{L^\infty(0,T;\dot{H}^{-\sigma}_{p_1})}\\
+\|\int^1_0[P_{>N}f'(\theta
u+(1-\theta)v)]\,d\theta\|_{L^\infty(0,T;L^{p_3})}\Big\}\|u-v\|_{L^2(0,T;\dot{H}^\sigma_6)}
\end{multline}

It follows from Lemma~\ref{Lchainrule} and the Sobolev Embedding
$\dot{H}^s\hookrightarrow L^\frac{6}{3-2s}$ that
\begin{multline}
\|\int^1_0[P_{>N}f'(\theta
u+(1-\theta)v)]\,d\theta\|_{L^\infty(0,T;\dot{H}^{-\sigma}_{p_1})}\\
+\|\int^1_0[P_{>N}f'(\theta
u+(1-\theta)v)]\,d\theta\|_{L^\infty(0,T;L^{p_3})}\lesssim\Big(\|u\|_{L^\infty(0,T;\dot{H}^s)}+\|v\|_{L^\infty(0,T;\dot{H}^s})\Big)^\alpha
\end{multline}
Note that $u,v\in C([0,T],\dot{H}^s)$, so we can find a uniform
$N_0$ independent on time such that when $N>N_0$,
\begin{multline}\label{n3iq2}
C\Big(\|\int^1_0[P_{>N}f'(\theta
u+(1-\theta)v)]\,d\theta\|_{L^\infty(0,T;\dot{H}^{-\sigma}_{p_1})}\\
+\|\int^1_0[P_{>N}f'(\theta
u+(1-\theta)v)]\,d\theta\|_{L^\infty(0,T;L^{p_3})}\Big)\leq\frac{1}{2}.
\end{multline}
For the term (I), by using the same method applied for the proof of
Lemma~\ref{bilinear est} and Bernstein inequality, we obtain
\begin{multline}
\|(P_{\leq
N}f')(u-v)\|_{\dot{H}^\sigma_{6/5}}\lesssim\|u-v\|_{\dot{H}^\sigma_3}(\|P_{\leq
N}f'\|_{\dot{H}^{-\sigma}_{\frac{6}{5-2s}}}+\|P_{\leq
N}f'\|_{L^2})\\
\lesssim N^\frac{1}{2}\|u-v\|_{\dot{H}^\sigma_3}(\|P_{\leq
N}f'\|_{\dot{H}^{-\sigma}_{p_1}}+\|P_{\leq N}f'\|_{L^{p_3}})
\end{multline}
By Lemma~\ref{Lchainrule}, the Sobolev Embedding
$\dot{H}^s\hookrightarrow L^\frac{6}{3-2s}$ and H\"older's
inequality on time, it follows that
\begin{equation}\label{n3iq1}
\|(I)\|_{L^2(0,T;\dot{H}^\sigma_{6/5})}\lesssim
T^\frac{1}{4}N^\frac{1}{2}
\Big(\|u\|_{L^\infty(0,T;\dot{H}^s)}+\|v\|_{L^\infty(0,T;\dot{H}^s)}\Big)^\alpha\|u-v\|_{L^4(0,T;\dot{H}^\sigma_3)}.
\end{equation}
Then by \eqref{n3iq}, \eqref{n3iq1}, \eqref{n3iq3} and
\eqref{n3iq2}, we  have
\begin{equation}
\|u-v\|_{L^2(0,T;\dot{H}^\sigma_6)}+\|u-v\|_{L^4(0,T;\dot{H}^\sigma_3)}\leq\frac{1}{2}\|u-v\|_{L^2(0,T;\dot{H}^\sigma_6)}+CT^\frac{1}{4}N^\frac{1}{2}
\|u-v\|_{L^4(0,T;\dot{H}^\sigma_3)}.
\end{equation}
If we choose $T$ small enough so that
$CT^\frac{1}{4}N^\frac{1}{2}<\frac{1}{4}$, then the right hand side
can be absorbed by the left hand side, which shows the unconditional
uniqueness.

\subsection{The case of $n\geq4,\ \alpha=\frac{4}{n-2s}$}

\subsubsection{\textbf{The boundness of the norm of $u-v$}}
Suppose $(a,b)$ and $(\lambda,\frac{2n}{n-2\sigma-4+2s})$ are
$\frac{n}{2}$-acceptable pairs. By Sobolev Embedding
$L^\frac{2n}{4+n-2s}\hookrightarrow
\dot{H}^\sigma_{\frac{2n}{2\sigma+4+n-2s}}$,
$\dot{H}^s\hookrightarrow L^\frac{2n}{n-2s}$, H\"older inequality
and $|f(u)-f(v)|\lesssim(|u|^\alpha+|v|^\alpha)(u-v)$, we can get
\begin{equation}\label{n4bd1}
\|f(u)-f(v)\|_{\dot{H}^\sigma_{\frac{2n}{2\sigma+4+n-2s}}}\lesssim\|u,v\|_{\dot{H}^s}^{\alpha+1}.
\end{equation}
If we can show
\begin{equation}\label{n4bd2}
\|u-v\|_{L^a(0,T;\dot{H}^\sigma_b)}\lesssim\|f(u)-f(v)\|_{L^{\lambda'}(0,T;\dot{H}^\sigma_{\frac{2n}{2\sigma+4+n-2s}})},
\end{equation}
then toghter with \eqref{n4bd1} and $u,v\in
L^\infty(0,T;\dot{H}^s)$, we obtain the boundness of
$\|u-v\|_{L^a(0,T;\dot{H}^\sigma_b)}$.

In order that \eqref{n4bd2} holds, by the non-sharp case of
Lemma~\ref{non-str-est}, the conditions
\eqref{non-str-est-cond1}-\eqref{non-str-est-cond3} have to be
fulfilled besides $(a,b)$ and $(\lambda,\frac{2n}{n-2\sigma-4+2s})$
being $\frac{n}{2}$-acceptable pairs.

According to the computation, if
$s-2<\sigma<\min\{0,s-\frac{n}{2(n-1)}\}$, we can find $(a,b)$ and
$(\lambda,\frac{2n}{n-2\sigma-4+2s})$ that satisfy all above
conditions. Here we list the restrictions on $b$, which are useful
for the following estimates
\begin{eqnarray}\label{conditions on b}
\left\{\begin{array}{ll}
 \frac{2\sigma+n-2s}{2n}<\frac{1}{b}<\frac{1}{2},\\
 \frac{(n-2)(n-2\sigma-4+2s)}{2n^2}\leq\frac{1}{b}\leq\frac{n-2\sigma-4+2s}{2(n-2)}.
 \end{array}
 \right.
\end{eqnarray}
\subsubsection{\textbf{The case of $n=4,5$, $\alpha=\frac{4}{n-2s}\geq1$ and $\frac{n}{4(n-1)}<s<1$}}

Suppose $(\gamma,\rho)$ is an $\frac{n}{2}$-acceptable pair with
$\frac{1}{\rho}=\frac{\sigma}{n}+\frac{1}{2}-\frac{s}{n}$, and
$\sigma$ satisfies the conditions
\begin{equation}\label{sigmaconditions}
-s\leq\sigma<0,\qquad
s-\frac{3n-4}{2(n-1)}<\sigma<s-\frac{n}{2(n-1)}.
\end{equation}

If the conditions in \eqref{sigmaconditions} hold, then by
Lemma~\ref{non-str-est}, we can find
$(\gamma,\rho),\ (a,b)$ and $\lambda$ such that
\begin{equation}\label{n4bd3}
\|u-v\|_{L^\gamma(0,T;\dot{H}^\sigma_\rho)}+\|u-v\|_{L^a(0,T;\dot{H}^\sigma_b)}\lesssim\|f(u)-f(v)\|_{L^{\lambda'}(0,T;\dot{H}^\sigma_{\frac{2n}{2\sigma+4+n-2s}})}.
\end{equation}
By the condition \eqref{non-str-est-cond1}, one can see
$\frac{1}{\gamma}+\frac{1}{\lambda}=\frac{n}{2}(1-\frac{1}{\rho}-\frac{n-2\sigma-4+2s}{2n})=1$.

We denote that
\begin{multline}
f(u)-f(v)=\Big[\int^1_0f'(\theta
u+(1-\theta)v)\,d\theta\Big](u-v)\\
=\Big\{\int^1_0[P_{\leq N}f'(\theta
u+(1-\theta)v)]\,d\theta\Big\}(u-v)+\Big\{\int^1_0[P_{>N}f'(\theta
u+(1-\theta)v)]\,d\theta\Big\}(u-v)\\
:=(I)+(II).
\end{multline}

For the term (II), by using Lemma~\ref{bilinear est} with
$\frac{1}{p_1}=\frac{2}{n}-\frac{\sigma}{n}$ and
$\frac{1}{p_3}=\frac{2}{n}$ and \eqref{sigmaconditions}, we can
obtain
\begin{multline}\label{n4bd4}
\|(II)\|_{L^{\lambda'}(0,T;\dot{H}^\sigma_{\frac{2n}{2\sigma+4+n-2s}})}\lesssim\Big\{\|\int^1_0[P_{>N}f'(\theta
u+(1-\theta)v)]\,d\theta\|_{L^\infty(0,T;\dot{H}^{-\sigma}_{p_1})}\\
+\|\int^1_0[P_{>N}f'(\theta
u+(1-\theta)v)]\,d\theta\|_{L^\infty(0,T;L^{p_3})}\Big\}\|u-v\|_{L^\gamma(0,T;\dot{H}^\sigma_\rho)}.
\end{multline}

Lemma~\ref{Lchainrule}, \eqref{sigmaconditions}, Sobolev Embedding
$\dot{H}^s\hookrightarrow L^\frac{2n}{n-2s}$ and
$\dot{H}^s\hookrightarrow\dot{H}^{-\sigma}_{\frac{2n}{n-2(s+\sigma)}}$
deduce that
\begin{multline}
\|\int^1_0[P_{>N}f'(\theta
u+(1-\theta)v)]\,d\theta\|_{L^\infty(0,T;\dot{H}^{-\sigma}_{p_1})}\\
+\|\int^1_0[P_{>N}f'(\theta
u+(1-\theta)v)]\,d\theta\|_{L^\infty(0,T;L^{p_3})}\lesssim\Big(\|u\|_{L^\infty(0,T;\dot{H}^s)}+\|v\|_{L^\infty(0,T;\dot{H}^s)}\Big)^\alpha
\end{multline}
Since $u,v\in C([0,T],\dot{H}^s)$, then we can find a uniform $N_0$
independent on time such that when $N>N_0$,
\begin{multline}\label{n4bd5}
C\Big(\|\int^1_0[P_{>N}f'(\theta
u+(1-\theta)v)]\,d\theta\|_{L^\infty(0,T;\dot{H}^{-\sigma}_{p_1})}\\
+\|\int^1_0[P_{>N}f'(\theta
u+(1-\theta)v)]\,d\theta\|_{L^\infty(0,T;L^{p_3})}\Big)\leq\frac{1}{2}.
\end{multline}

For the term (I), by using the same method applied for the proof of
Lemma~\ref{bilinear est}, we  have
\begin{equation}
\|(P_{\leq
N}f')(u-v)\|_{\dot{H}^\sigma_{r'}}\lesssim\|u-v\|_{\dot{H}^\sigma_b}(\|P_{\leq
N}f'\|_{\dot{H}^{-\sigma}_{x_1}}+\|P_{\leq N}f'\|_{L^{x_2}}),
\end{equation}
where$\frac{1}{x_1}=\frac{1}{2}-\frac{s}{n}+\frac{2}{n}-\frac{1}{b}$
and
$\frac{1}{x_2}=\frac{1}{2}+\frac{\sigma}{n}-\frac{s}{n}+\frac{2}{n}-\frac{1}{b}$.

If $\sigma$ satisfies \eqref{sigmaconditions}, then the conditions
on $b$ in \eqref{conditions on b} hold. So by Bernstein's
inequality, one has
\begin{eqnarray}
\|P_{\leq N}f'\|_{\dot{H}^{-\sigma}_{x_1}}&\lesssim&
N^{n(\frac{1}{b}-\frac{2\sigma-2s+n}{2n})}\|P_{\leq
N}f'\|_{\dot{H}^{-\sigma}_{p_1}};\\
\|P_{\leq N}f'\|_{L^{x_2}}&\lesssim&
N^{n(\frac{1}{b}-\frac{2\sigma-2s+n}{2n})}\|P_{\leq
N}f'\|_{L^{p_3}}.
\end{eqnarray}

By Lemma~\ref{Lchainrule}, \eqref{sigmaconditions}, Sobolev
Embedding $\dot{H}^s\hookrightarrow L^\frac{2n}{n-2s}$,
$\dot{H}^s\hookrightarrow\dot{H}^{-\sigma}_{\frac{2n}{n-2(s+\sigma)}}$
and H\"older inequality on time, we have
\begin{multline}\label{n4bd6}
\|(I)\|_{L^{\lambda'}(0,T;\dot{H}^\sigma_{\frac{2n}{2\sigma+4+n-2s}})}\\
\leq
CT^{1-\frac{1}{\lambda}-\frac{1}{a}}N^{n(\frac{1}{b}-\frac{2\sigma-2s+n}{2n})}
\Big(\|u\|_{L^\infty(0,T;\dot{H}^s)}+\|v\|_{L^\infty(0,T;\dot{H}^s)}\Big)^\alpha\|u-v\|_{L^a(0,T;\dot{H}^\sigma_b)}\\
\leq\frac{1}{4}\|u-v\|_{L^a(0,T;\dot{H}^\sigma_b)},
\end{multline}
if $T$ is small enough such that
\[
CT^{1-\frac{1}{\lambda}-\frac{1}{a}}N^{n(\frac{1}{b}-\frac{2\sigma-2s+n}{2n})}
\Big(\|u\|_{L^\infty(0,T;\dot{H}^s)}+\|v\|_{L^\infty(0,T;\dot{H}^s)}\Big)^\alpha<\frac{1}{4}.
\]

In conclusion, by \eqref{n4bd3},\eqref{n4bd4},\eqref{n4bd5} and
\eqref{n4bd6}, we can have
\begin{equation}
\|u-v\|_{L^\gamma(0,T;\dot{H}^\sigma_\rho)}+\|u-v\|_{L^a(0,T;\dot{H}^\sigma_b)}<\frac{3}{4}(\|u-v\|_{L^\gamma(0,T;\dot{H}^\sigma_\rho)}+\|u-v\|_{L^a(0,T;\dot{H}^\sigma_b)}),
\end{equation}
which shows the unconditional uniqueness.

\subsubsection{\textbf{The case of $n\geq5$, $\alpha=\frac{4}{n-2s}<1$ and $s_0<s<1$}}

The proof is similar to that for $n=4,5$ with $\alpha\geq1$,
except that we apply Lemma~\ref{FcrfaHcf} instead of
Lemma~\ref{Lchainrule}.

We still suppose $(\gamma,\rho)$ is a $\frac{n}{2}$-acceptable pair with $\frac{1}{\rho}=\frac{\sigma}{n}+\frac{1}{2}-\frac{s}{n}$, and
$\sigma$ satisfies the conditions
\begin{equation}\label{sigmaconditions*}
\frac{-4s}{n-2s}<\sigma<0,\qquad
s-\frac{3n-4}{2(n-1)}<\sigma<s-\frac{n}{2(n-1)}
\end{equation}

By
Lemma~\ref{non-str-est}, the conditions in \eqref{sigmaconditions*} can help us to find
$(\gamma,\rho),\ (a,b)$ and $\lambda$ such that
\begin{equation}\label{n6bd1}
\|u-v\|_{L^\gamma(0,T;\dot{H}^\sigma_\rho)}+\|u-v\|_{L^a(0,T;\dot{H}^\sigma_b)}\lesssim\|f(u)-f(v)\|_{L^{\lambda'}(0,T;\dot{H}^\sigma_{\frac{2n}{2\sigma+4+n-2s}})},
\end{equation}
where
$\frac{1}{\gamma}+\frac{1}{\lambda}=\frac{n}{2}(1-\frac{1}{\rho}-\frac{n-2\sigma-4+2s}{2n})=1$
from the condition \eqref{non-str-est-cond1}.

We denote that
\begin{multline}
f(u)-f(v)=\Big[\int^1_0f'(\theta
u+(1-\theta)v)\,d\theta\Big](u-v)\\
=\Big\{\int^1_0[P_{\leq N}f'(\theta
u+(1-\theta)v)]\,d\theta\Big\}(u-v)+\Big\{\int^1_0[P_{>N}f'(\theta
u+(1-\theta)v)]\,d\theta\Big\}(u-v)\\
:=(I)+(II).
\end{multline}

For the term (II), from Lemma~\ref{bilinear est} with
$\frac{1}{p_1}=\frac{2}{n}-\frac{\sigma}{n}$, and
$\frac{1}{p_3}=\frac{2}{n}$, and \eqref{sigmaconditions*}, one can get
\begin{multline}\label{n6bd2}
\|(II)\|_{L^{\lambda'}(0,T;\dot{H}^\sigma_{\frac{2n}{2\sigma+4+n-2s}})}\lesssim\Big\{\|\int^1_0[P_{>N}f'(\theta
u+(1-\theta)v)]\,d\theta\|_{L^\infty(0,T;\dot{H}^{-\sigma}_{p_1})}\\
+\|\int^1_0[P_{>N}f'(\theta
u+(1-\theta)v)]\,d\theta\|_{L^\infty(0,T;L^{p_3})}\Big\}\|u-v\|_{L^\gamma(0,T;\dot{H}^\sigma_\rho)}
\end{multline}

By using Lemma~\ref{FcrfaHcf}, \eqref{sigmaconditions*} and Sobolev Embedding
$\dot{H}^s\hookrightarrow L^\frac{2n}{n-2s}$, we can obtain
\begin{multline}
\|\int^1_0[P_{>N}f'(\theta
u+(1-\theta)v)]\,d\theta\|_{L^\infty(0,T;\dot{H}^{-\sigma}_{p_1})}\\
+\|\int^1_0[P_{>N}f'(\theta
u+(1-\theta)v)]\,d\theta\|_{L^\infty(0,T;L^{p_3})}\lesssim\Big(\|u\|_{L^\infty(0,T;\dot{H}^s)}+\|v\|_{L^\infty(0,T;\dot{H}^s})\Big)^\alpha
\end{multline}
Since $u,v\in C([0,T],\dot{H}^s)$, then one can find a uniform $N_0$
independent on time such that when $N>N_0$,
\begin{multline}\label{n6bd4}
C\Big(\|\int^1_0[P_{>N}f'(\theta
u+(1-\theta)v)]\,d\theta\|_{L^\infty(0,T;\dot{H}^{-\sigma}_{p_1})}\\
+\|\int^1_0[P_{>N}f'(\theta
u+(1-\theta)v)]\,d\theta\|_{L^\infty(0,T;L^{p_3})}\Big)\leq\frac{1}{2}.
\end{multline}

For the term (I), we use the same argument as that in the case of $n=4,5$ and
$\alpha\geq1$ to get
\begin{multline}\label{n6bd3}
\|(I)\|_{L^{\lambda'}(0,T;\dot{H}^\sigma_{\frac{2n}{2\sigma+4+n-2s}})}\\
\leq
CT^{1-\frac{1}{\lambda}-\frac{1}{a}}N^{n(\frac{1}{b}-\frac{2\sigma-2s+n}{2n})}
\Big(\|\int^1_0[P_{>N}f'(\theta
u+(1-\theta)v)]\,d\theta\|_{L^\infty(0,T;\dot{H}^{-\sigma}_{p_1})}\\
+\|\int^1_0[P_{>N}f'(\theta
u+(1-\theta)v)]\,d\theta\|_{L^\infty(0,T;L^{p_3})}\Big)\|u-v\|_{L^a(0,T;\dot{H}^\sigma_b)}\\
\leq C
T^{1-\frac{1}{\lambda}-\frac{1}{a}}N^{n(\frac{1}{b}-\frac{2\sigma-2s+n}{2n})}
\Big(\|u\|_{L^\infty(0,T;\dot{H}^s)}+\|v\|_{L^\infty(0,T;\dot{H}^s)}\Big)^\alpha\|u-v\|_{L^a(0,T;\dot{H}^\sigma_b)}\\
\leq\frac{1}{4}\|u-v\|_{L^a(0,T;\dot{H}^\sigma_b)},
\end{multline}
if $T$ is sufficiently small such that
\[
CT^{1-\frac{1}{\lambda}-\frac{1}{a}}N^{n(\frac{1}{b}-\frac{2\sigma-2s+n}{2n})}
\Big(\|u\|_{L^\infty(0,T;\dot{H}^s)}+\|v\|_{L^\infty(0,T;\dot{H}^s)}\Big)^\alpha<\frac{1}{4}.
\]

In conclusion, by using \eqref{n6bd1},\eqref{n6bd2},\eqref{n6bd3} and
\eqref{n6bd4}, we  have
\begin{equation}
\|u-v\|_{L^\gamma(0,T;\dot{H}^\sigma_\rho)}+\|u-v\|_{L^a(0,T;\dot{H}^\sigma_b)}<\frac{3}{4}(\|u-v\|_{L^\gamma(0,T;\dot{H}^\sigma_\rho)}+\|u-v\|_{L^a(0,T;\dot{H}^\sigma_b)}),
\end{equation}
which shows the unconditional uniqueness, and we complete the proof
of Theorem~1.6.

\section*{Acknowledgement}
The authors would like to thank Professor T. Cazenave and Professor
Y. Tsutsumi for their helpful suggestions. This work is supported partially by NSFC
10871175, 10931007, and Zhejiang NSF of China Z6100217,

\begin{figure}
\centering
\includegraphics[width=0.8\textwidth,height=0.4\textheight]{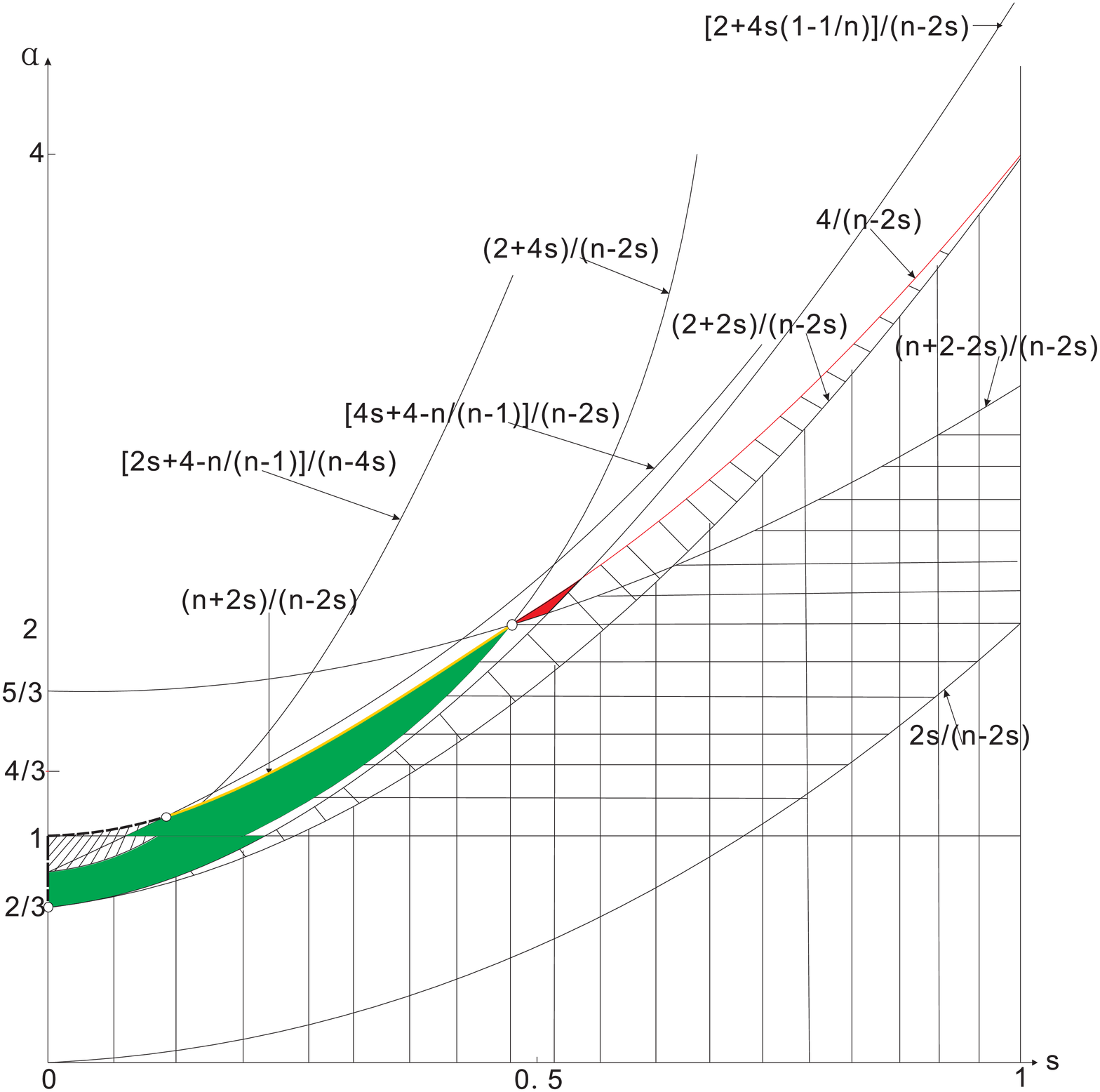}
\caption{Case $n=3$} Kato--Vertical route area; Furioli and Terraneo--
Horizontal area; Rogers--Oblique line; Win and Tsutsumi--Red color;
Open parts-Left slashes and thick dashed lines; Beside to cover the known areas, the new part 
of our results--Green color for the subcritical case and
yellow one for the critical cases.
\includegraphics[width=0.8\textwidth,height=0.4\textheight]{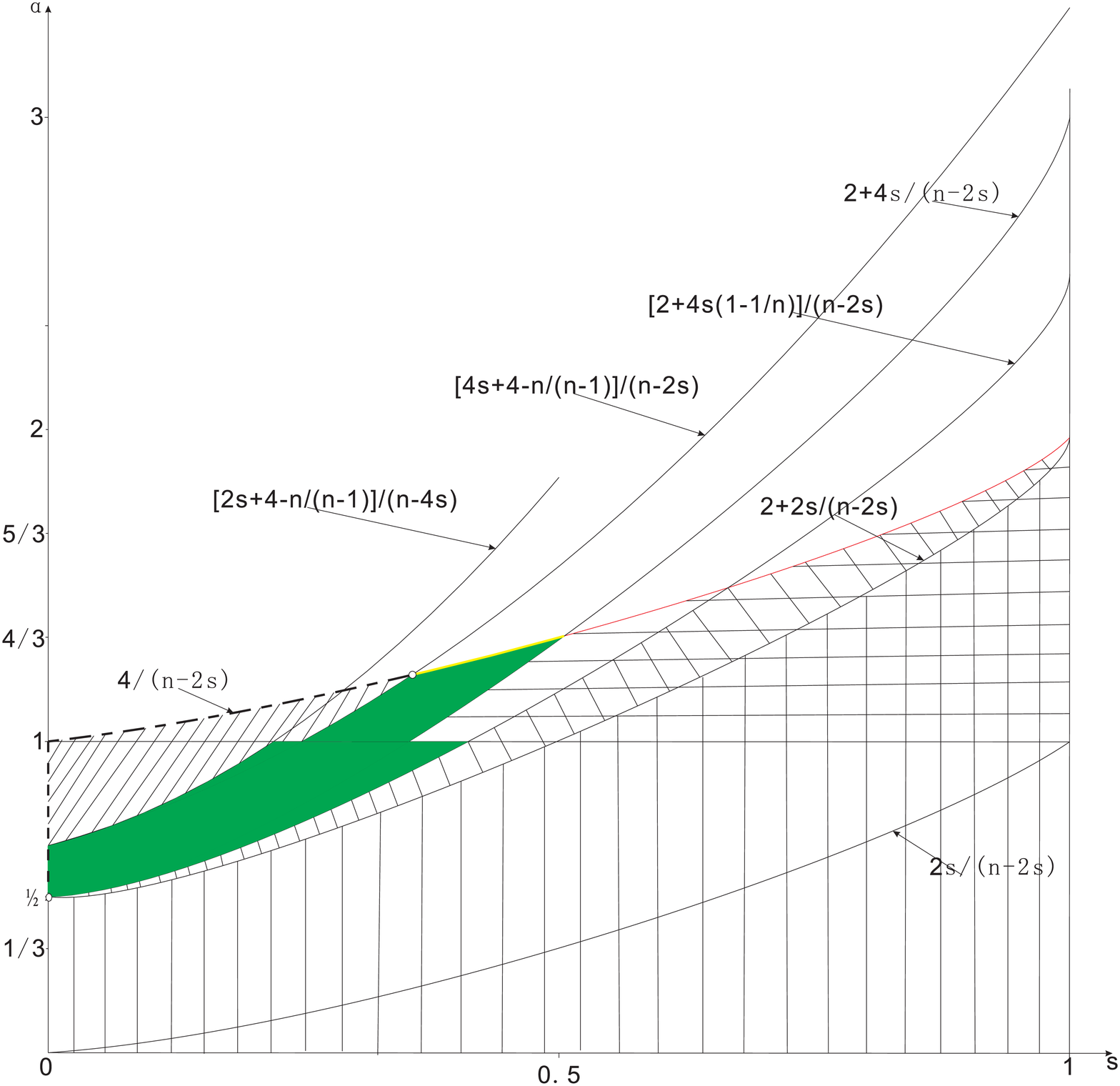}
\caption{Case $n=4$}Kato--Vertical route area; Furioli and Terraneo--
Horizontal area; Rogers--Oblique line; Win and Tsutsumi--Red color;
Open parts-Left slashes and thick dashed lines; Beside to cover the known areas, the new part
of our results--Green color for the subcritical case and
yellow one for the critical cases.
 \end{figure}
 \begin{figure}
 \centering
\includegraphics[width=0.8\textwidth,height=0.4\textheight]{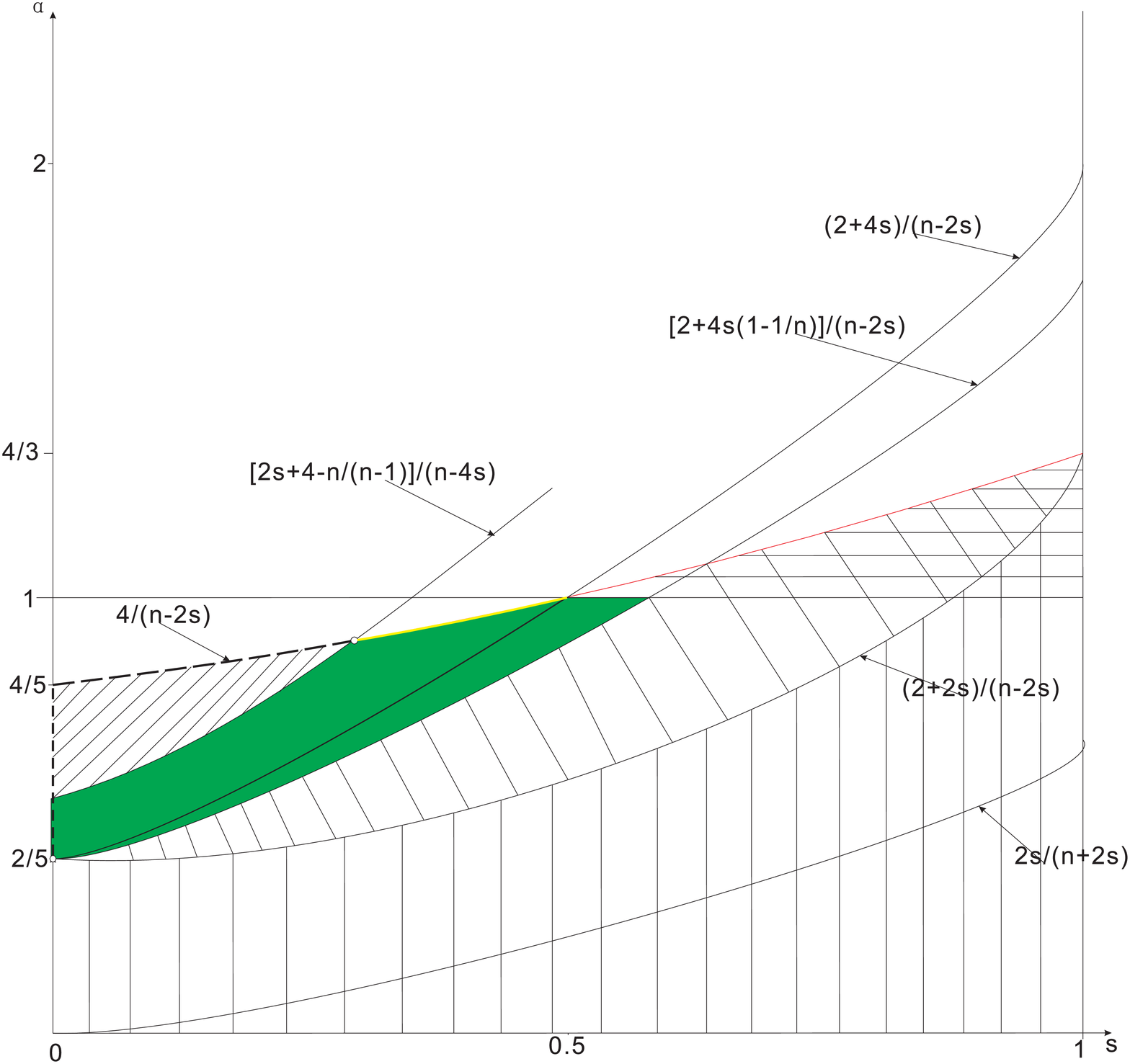}
\caption{Case $n=5$} Kato--Vertical area; Furioli and Terraneo--
Horizontal area; Rogers--Oblique line; Win and Tsutsumi--Red color;
Open parts-Left slashes and thick dashed lines; Beside to cover the known areas, the new part
of our results--Green color for the subcritical case and
yellow one for the critical cases.
 \centering
\includegraphics[width=0.8\textwidth,height=0.4\textheight]{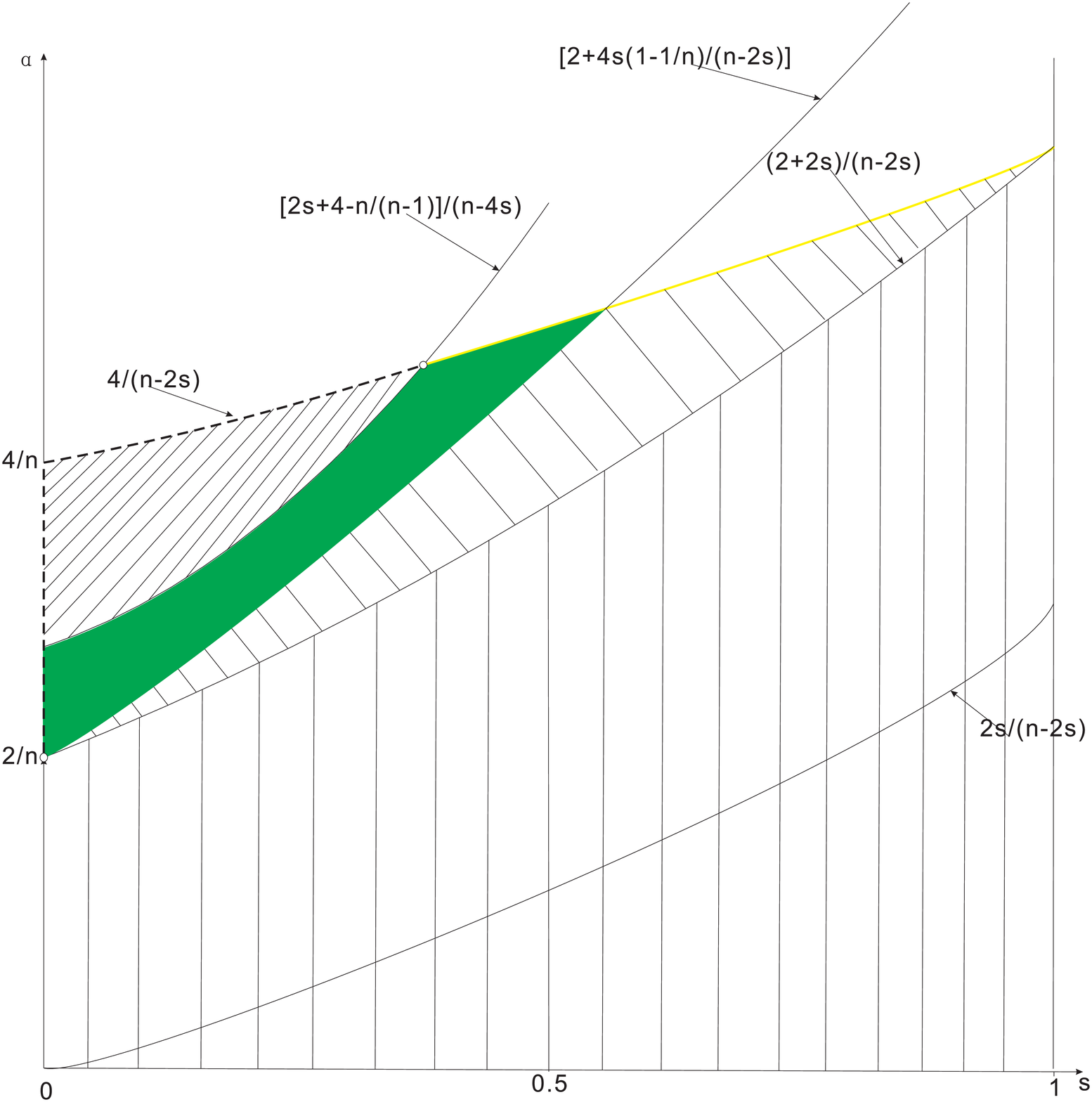}
\caption{Cases $n\geq6$}Kato--Vertical route area; Rogers--Oblique line;
Open parts-Left slashes and thick dashed lines; Beside to cover the known areas, the new part
of our results--Green color for the subcritical case and
yellow one for the critical cases.
\end{figure}

\end{document}